\documentclass[a4paper,10pt,psamsfonts]{amsart}

\usepackage[utf8]{inputenc} 

\usepackage{amsmath}
\usepackage{amstext}
\usepackage{amsfonts}
\usepackage{amsthm}
\usepackage{amssymb}
\usepackage[bf,SL,BF]{subfigure}
\usepackage{graphicx}
\usepackage{caption}
\usepackage{subfigure}
\usepackage{subfloat}
\usepackage{algorithm}
\usepackage[noend]{algpseudocode}
\usepackage{listings}
\usepackage{hyperref}
\hypersetup{colorlinks}

\usepackage{color}

\usepackage{changes}

\newcommand{\R}{{\mathbb{R}}}

\newcommand{\E}{{\mathbb{E}}}
\newcommand{\N}{{\mathbb{N}}}
\newcommand{\D}{{\mathcal{D}}}
\newcommand{\F}{{\mathcal{F}}} 
\renewcommand{\P}{{\mathbb{P}}} 

\newcommand{\B}{{\mathcal{B}}}

\newcommand{\diff}[1]{\,\mathrm{d}#1}

\newcommand{\one}{\mathbb{I}}

\newcommand{\triple}{{\vert\kern-0.25ex\vert\kern-0.25ex\vert}}

\DeclareMathOperator*{\esssup}{ess\,sup}

\theoremstyle{plain}

\newtheorem{definition}{Definition}[section]
\newtheorem{theorem}[definition]{Theorem}
\newtheorem{lemma}[definition]{Lemma}
\newtheorem{corollary}[definition]{Corollary}

\newtheorem{assumption}[definition]{Assumption}

\theoremstyle{definition}
\newtheorem{remark}[definition]{Remark}

\begin{document}

\title[Randomized Quadrature for the Finite Element Method]
{Application of Randomized Quadrature Formulas\\
to the Finite Element Method\\
for Elliptic Equations\\
}

\author[R.~Kruse]{Raphael Kruse}
\address{Raphael Kruse\\
Institut f\"ur Mathematik\\
Technische Universit\"at Berlin\\
Stra\ss e des 17.~Juni 136\\
DE-10623 Berlin\\
Germany}
\email{kruse@math.tu-berlin.de}

\author[N.~Polydorides]{Nick Polydorides}
\address{Nick Polydorides\\
School of Engineering\\
University of Edinburgh\\
Edinburgh, EH9 3FB\\
UK;
The Alan Turing Institute, London, UK}
\email{n.polydorides@ed.ac.uk}

\author[Y.~Wu]{Yue Wu}
\address{Yue Wu\\
Mathematical Institute\\
University of Oxford\\
Oxford, OX2 6GG\\
UK}
\email{yue.wu@maths.ox.ac.uk}


\begin{abstract}
  The implementation of the finite element method for linear elliptic equations
  requires to assemble the stiffness matrix and the load vector. In general,
  the entries of this matrix-vector system are not known explicitly but need to
  be approximated by quadrature rules. If the coefficient functions
  of the differential operator or the forcing term are irregular, then standard
  quadrature formulas, such as the barycentric quadrature rule, may not be 
  reliable. In this paper we investigate the application of two randomized
  quadrature formulas to the finite element method for such 
  elliptic boundary value problems with irregular coefficient functions. 
  We give a detailed error analysis of these methods, discuss their
  implementation, and demonstrate their capabilities in several numerical
  experiments.  
\end{abstract}

\maketitle


\section{Introduction}
\label{sec:intro}

Let $\D \subset \R^2$ be a convex, bounded, and polygonal domain. We consider a
linear elliptic boundary value problem of the following form:
Find a mapping $u \colon \D \to \R$ such that 
\begin{align}
  \label{eq:BVP}
  \begin{cases}
    -\mathrm{div}\big( \sigma \nabla u \big) = f,&
    \text{ in } \D,\\
    u = 0, & \text{ on } \partial \D,
  \end{cases}
\end{align}
where $\sigma, f \colon \D \to \R$ are given coefficient functions with
$\sigma(x) \ge \sigma_0 > 0$ for all $x \in \D$. 
Provided $\sigma$ is globally bounded and $f$ is square-integrable,
it is well-known that \eqref{eq:BVP} admits a unique solution $u \in H^1_0(\D)$
in the weak sense satisfying
\begin{align}
  \label{eq:weak_sol}
  \int_\D \sigma(x) \nabla u(x) \cdot \nabla v(x) \diff{x} 
  = \int_\D f(x) v(x) \diff{x} 
\end{align}
for all $v \in H^1_0(\D)$. Here, we denote by $H^1_0(\D)$ the Sobolev space of
weakly differentiable and square-integrable functions which (in some sense)
satisfy the homogeneous Dirichlet boundary condition. 
In Section~\ref{sec:notation} we provide more details on
the function spaces used throughout this paper.
We also refer, for instance, to
\cite[Chapters~8--9]{brezis2011} or \cite[Chapter~6]{evans2010} for an
introduction to the variational formulation of elliptic boundary value
problems of the form \eqref{eq:BVP}.  

Elliptic equations such as \eqref{eq:BVP} appear in many applications, e.g., in
mechanical engineering and physics. It is also an intensively studied problem to
introduce the Galerkin finite element method as found in many text books in
numerical analysis, e.g. \cite{brenner2008, larson2013, larsson2009,
strang1973}. In the same spirit, we use \eqref{eq:BVP} as a model problem
to demonstrate the applicability of randomized quadrature formulas to the finite
element method. 

To this end, we consider a family $(\mathcal{T}_h)_{h \in (0,1]}$ of finite  
subdivisions of the polygonal domain $\D \subset \R^2$ into 
triangles.
Hereby, the parameter $h \in (0,1]$ denotes the maximal
edge length of the elements in $\mathcal{T}_h$. 
For every partition $\mathcal{T}_h$ we define
$S_h \subset H^1_0(\D)$ as the associated 
finite element space consisting of piecewise linear functions.

Then, we obtain an approximation of
the exact solution to the boundary value problem \eqref{eq:BVP}  
by solving the following finite dimensional problem:
For $h \in (0,1]$ find $u_h \in S_h$ satisfying
\begin{align}
  \label{eq:Galerkin}
  \int_\D \sigma(x) \nabla u_h(x) \cdot \nabla v_h(x) \diff{x}
  = \int_\D f(x) v_h(x) \diff{x} 
\end{align}
for all $v_h \in S_h$.
For the practical computation of the approximation $u_h \in S_h$,
it is then convenient to rewrite \eqref{eq:Galerkin} as a
system of linear equations. More precisely,
let $(\varphi_j)_{j = 1}^{N_h}$ be a basis of $S_h$,
where $N_h = \dim(S_h)$ denotes the number of degrees of freedom.
Then, we have the representation
\begin{align*}
  u_h = \sum_{j = 1}^{N_h} u_j \varphi_j,
\end{align*}
where the entries of the vector $\mathbf{u} = [u_1,\ldots,u_{N_h}]^{\top}
\in\R^{N_h}$ are yet to be determined. After inserting this 
representation of $u_h$ into the finite dimensional problem 
\eqref{eq:Galerkin} and by testing with all basis functions
$(\varphi_j)_{j = 1}^{N_h}$ we arrive at a system of linear equations. In
matrix-vector form this system is written as 
\begin{align}
  \label{eq:matvec}
  A_h \mathbf{u} = f_h,
\end{align}
where the \emph{stiffness matrix} $A_h \in \R^{N_h \times N_h}$ is given by 
\begin{align*}
  [A_h]_{i,j} = \int_\D \sigma(x) \nabla \varphi_i(x) \cdot \nabla \varphi_j
  (x) \diff{x} 
\end{align*}
for all $i,j \in \{1,\ldots,N_h\}$. Moreover, the \emph{load vector} $f_h \in
\R^{N_h}$ has the entries
\begin{align}
  \label{eq:load}
  [f_h]_i = \int_\D f(x) \varphi_i(x) \diff{x}, \quad i \in \{1,\ldots, N_h\}.
\end{align}
If, on the one hand, the entries of $A_h$ and $f_h$ are known explicitly, it is
straight-forward to use standard solvers for the linear system
\eqref{eq:matvec} in order to determine $\mathbf{u} \in \R^{N_h}$ and, hence,
$u_h \in S_h$ numerically. For instance, we refer to the monograph
\cite{hackbusch2016} for an overview of suitable solvers. 

On the other hand, for general $\sigma \in L^\infty(\D)$ and $f \in L^2(\D)$,
the entries of the stiffness matrix and the load vector are often not
computable explicitly. Such irregular coefficients often appear in problems in
uncertainty quantification to model incomplete knowledge of the problem
parameters. See \cite{barth2018} and the references therein.
In the literature, the reader is
advised to approximate the entries by suitable quadrature formulas.
For instance, we refer to \cite[Section~5.6]{larsson2009} and
\cite[Section~4.3]{strang1973}.

However, standard methods for numerical integration,
such as the trapezoidal sum, require point evaluations of the coefficient
functions $\sigma$ and $f$. Therefore, these quadrature formulas
are, in general, only applicable if additional smoothness requirements, such as
continuity, are imposed on $\sigma$ and $f$.   
The purpose of this paper is to show that this problem can be circumvented if 
we approximate the entries of $A_h$ and $f_h$ by \emph{randomized quadrature
formulas}. As it will turn out, these quadrature formulas do not require the
continuity of $f$ and $\sigma$.

Before we give a more detailed outline of the content of this paper, 
let us mention that we consider randomized quadrature formulas of a form 
that has originally been introduced by S.~Haber in \cite{haber1966,
haber1967,haber1969}. His important observation was that the accuracy of the
standard Monte Carlo method can be increased drastically, if the random
sampling points are distributed more evenly over the integration domain. More
precisely, he proposed to place the random sampling points in disjoint 
subdomains whose volumes decay asymptotically with the number of samples. 
If the integrand possesses more regularity than being merely square-integrable
this approach reduces the variance of the randomized quadrature formula
significantly. In particular, one often observes an higher order of convergence
compared to standard Monte Carlo estimators or purely deterministic methods. 
For more details on this line of arguments we also refer to the proof of
Lemma~\ref{lem:QMC} further below. Moreover, related results are found in
\cite{cambanis1992, masry1990}. 

More recently, it has been shown that such randomized quadrature formulas
are also applicable to the numerical approximation of ordinary differential
equations with time-irregular coefficient functions. We refer, for instance, 
to \cite{daun2011, heinrich2008, jentzen2009, kruse2017, stengle1990,
stengle1995} for results on randomized one-step methods. Further, 
these methods have also been applied for the \emph{temporal
discretization} of evolution equations in infinite dimensions, 
see \cite{eisenmann2017, hofmanova2017}, and of stochastic differential
equations, see \cite{kruse2017b, przybylowicz2014}.

Besides \cite{heinrich2006}, where the information based complexity of
randomized algorithms for elliptic partial differential equations has been
investigated, it appears that the application of randomized quadrature formulas
to the \emph{spatial discretization} of boundary value problems is not
well-studied yet. 

In this paper, we first consider a stratified Monte Carlo
estimator in the spirit of \cite{haber1966}. More precisely,
the estimator defined in \eqref{eq:MC} below, is based on an admissible
triangulation $\mathcal{T}_h$ of $\D$ and exactly one uniformly distributed 
random point on each triangle of the triangulation.
We show in Section~\ref{sec:general} that this estimator gives
approximations of the entries in the stiffness matrix and the load vector,
which are convergent at least with order $1$ with respect to the
root-mean-square norm. Under slightly increased 
regularity assumptions, such as $f \in L^p(\D)$ with $p \in (2,\infty]$ and
$\sigma \in W^{s,q}(\D)$ with $s \in (0,1]$, $q \in (2,\infty]$,
we also show that the resulting randomized finite element solution $u_h^{MC}$
converges to the exact solution $u \in H^1_0(\D)$. The precise error estimate
is given in Theorem~\ref{thm:errorH1}.

In Section~\ref{sec:importanceSampling}, we propose an
importance sampling estimator for the approximation of the load vector.
Hereby, the random points are placed according to a non-uniform distribution,
whose probability density function is proportional to the basis functions of
the finite element space.  
The section also contains a detailed analysis of the error with respect
to the norms in $L^2(\D)$ and $H^1_0(\D)$, where we purely focus on the
associated finite element problem for the Poisson equation \eqref{eq:Poisson},
i.e. Equation~\eqref{eq:BVP} with $\sigma \equiv 1$. These results are
stated in Theorem~\ref{thm:errorISH1} and Theorem~\ref{thm:errorISL2}.

In Section~\ref{sec:sampling} we discuss the implementation of the 
randomized quadrature formulas. Essentially, this is achieved by a
transformation to a reference triangle, typically the $2$-simplex, and a
general rejection algorithm. Finally, we report on some numerical experiments 
in Section~\ref{sec:numexp}. 


\section{Notation and preliminaries}
\label{sec:notation}

In this section, we fix some notation and introduce several 
function spaces, which are used throughout this paper.
We also revisit the variational formulation of the boundary value
problem~\eqref{eq:BVP} and its approximation by the finite element method.
The section also contains a brief overview of some terminology from
probability. 

By $\N$ we denote the set of all positive integers, while $\N_0 := \N \cup
\{0\}$. As usual, the set $\R$ consists of all real numbers. By $| \cdot |$ we
denote the Euclidean norm on the Euclidean space $\R^d$ for any $d \in \N$. In
particular, if $d = 1$ then $| \cdot |$ coincides with taking
the absolute value.

Throughout this paper we often use $C$ as a generic constant, which may
vary from appearance to appearance. However, $C$ is not allowed to depend on
numerical parameters such as $h \in (0,1]$. 

Next, let us introduce some function spaces. Throughout this paper, we assume
that $\D \subset \R^2$ 
is a bounded, convex and polygonal domain.  
By $L^p(\D)$, $p \in [1,\infty]$, we denote the Banach 
space of (equivalence classes of) $p$-fold Lebesgue integrable functions, which
is endowed with the norm 
\begin{align*}
  \| f \|_{L^p(\D)} &= \Big( \int_\D |f(x)|^p \diff{x}
  \Big)^{\frac{1}{p}} \quad \text{ for } p \in [1,\infty),\\
  \| f \|_{L^\infty(\D)} &= \esssup_{x \in \D} |f(x)|.
\end{align*}
As it is customary, we do not distinguish notationally between functions and
their equivalence classes.

An important example of an element in $L^p(\D)$ for any value of $p \in
[1,\infty]$ is the indicator function of a measurable set $B \subseteq \D$
denoted by $\one_B$. This function fulfills $\one_B(x) = 1$ if $x \in B$, else
$\one_B(x)=0$.

Moreover, we denote by $W^{k,p}(\D) \subset L^p(\D)$, $p \in [1,\infty]$, $k
\in \N$, the Sobolev space with differentiation index $k$. To be more precise,
$W^{k,p}(\D)$ consists of all $p$-fold integrable functions that are $k$-times
partially differentiable in the weak sense and whose derivatives are also
$p$-fold integrable. If $W^{k,p}(\D)$ is endowed with the norm
\begin{align*}
  \| f \|_{W^{k,p}(\D)} &= \Big( \sum_{ \alpha \in \N_0^2, |\alpha|\le
  k} \| \partial^\alpha f \|_{L^p(\D)}^p \Big)^{\frac{1}{p}} \quad \text{ for
  } p \in [1,\infty),\\
  \| f \|_{W^{k,\infty}(\D)} &= \sum_{ \alpha \in \N_0^2, |\alpha|\le
  k} \| \partial^\alpha f \|_{L^\infty(\D)},
\end{align*}
then it is also a Banach space.
Here we make use of the standard multi-index notation for partial derivatives,
that is, for $\alpha \in \N_0^2$ we define $|\alpha| = \alpha_1 + \alpha_2$
and 
\begin{align*}
  \partial^\alpha f := \frac{\partial^{|\alpha|}}{\partial_{x_1}^{\alpha_1}
  \partial_{x_2}^{\alpha_2}} f.
\end{align*}
Further, if $p = 2$ then $L^2(\D)$ and $H^k(\D) := W^{k,2}(\D)$ are Hilbert
spaces.  
The inner products are denoted by $(\cdot, \cdot)_{L^2(\D)}$ 
and $(\cdot, \cdot)_{H^k(\D)}$, respectively.

In order to incorporate homogeneous Dirichlet boundary conditions, we also
introduce the space $H^1_0(\D)$, which is defined as the closure of the set of
all infinitely often differentiable functions with compact support in $\D$ with
respect to the norm in $H^1(\D)$, that is
\begin{align*}
  H^1_0(\D) := \overline{ C_{c}^\infty(\D)}^{\|\cdot\|_{H^1(\D)}}.
\end{align*}
It is well-known that the standard $H^1(\D)$-norm and the semi-norm
\begin{align*}
  | f |_{H^1(\D)} = \Big( \sum_{ i = 1}^2 \Big\| \frac{\partial}{\partial x_i}
  f \Big\|_{L^2(\D)}^2 \Big)^{\frac{1}{2}}
  = \Big( \int_\D | \nabla f |^2 \diff{x} \Big)^{\frac{1}{2}}
\end{align*}
are equivalent on $H^1_0(\D)$. In particular, the space
$(H^1_0(\D),|\cdot|_{H^1(\D)}, (\cdot,\cdot)_{H^1_0(\D)})$ is a separable
Hilbert space. For a detailed introduction to Sobolev spaces we refer the
reader, for instance, to \cite[Chapter~5]{evans2010}. 

For a domain $\D \subset \R^2$, $p \in [1,\infty)$, and $s \in (0,1)$
the Sobolev--Slobodeckij norm $\| \cdot\|_{W^{s,p}(\D)}$ 
is given by
\begin{align}
  \label{eq:fracSobol}
  \| f \|_{W^{s,p}(\D)} 
  = \Big( \| f \|^p_{L^p(\D)} + \int_{\D} \int_{\D} \frac{|f(x_1) -
  f(x_2) |^p}{|x_1 - x_2|^{2 + s p}} \diff{x_2} \diff{x_1}
  \Big)^{\frac{1}{p}}.
\end{align}
Then, the fractional order Sobolev space $W^{s,p}(\D)$ consists of all $f \in
L^p(\D)$ satisfying $\| f \|_{W^{s,p}(\D)}< \infty$. 
By $|\cdot|_{W^{s,p}(\D)}$ we denote the corresponding semi-norm, which only
consists of the double integral part in \eqref{eq:fracSobol}.
Further details on these spaces are found in \cite{dinezza2012}. 

Next, we revisit the variational formulation of the boundary
value problem \eqref{eq:BVP}.  If $\sigma \in L^\infty(\D)$, $\sigma(x) \ge
\sigma_0 > 0$ for almost every $x \in \D$, and $f \in L^2(\D)$, then it is 
well-known that the bilinear form $a \colon H^1_0(\D) 
\times H^1_0(\D) \to \R$ and the linear functional $F \colon H^1_0(\D) \to \R$
given by 
\begin{align}
  \label{eq:a_2d}
  a(u,v) &:= \int_\D \sigma(x) \nabla u(x) \cdot \nabla v(x) \diff{x},\\
  \label{eq:F_2d}
  F(v) &:= \int_\D f(x) v(x) \diff{x}
\end{align}
for all $u, v \in H^1_0(\D)$ are well-defined. Moreover, $a$ is strongly
positive and bounded, that is, it holds
\begin{align}
  \label{eq:a_pos}
  a(v,v) &\ge \sigma_0 |v|_{H^1(\D)}^2,\\
  \label{eq:a_bdd}
  |a (u,v) | &\le \|\sigma\|_{L^\infty(\D)} |u|_{H^1(\D)} |v|_{H^1(\D)}
\end{align}
for all $u, v \in H^1_0(\D)$.
Further, $F$ is a bounded linear functional.

Therefore, the lemma of Lax--Milgram, cf. \cite[Chapter~6]{evans2010}, is
applicable and ensures the existence of a unique weak solution $u \in
H^1_0(\D)$ satisfying 
\begin{align}
  \label{eq:varprob2d}
  a(u,v) = F(v) \quad \text{for all } v \in H^1_0(\D).
\end{align}
Observe that \eqref{eq:varprob2d} coincides with \eqref{eq:weak_sol}.

For the error analysis in Section~\ref{sec:general} and Section
\ref{sec:importanceSampling}, it will be necessary to 
impose the following additional regularity condition on the
exact solution.

\begin{assumption}
  \label{as:reg}
  The variational problem \eqref{eq:varprob2d} has a uniquely determined
  strong solution, i.e., the unique weak solution $u$ to
  \eqref{eq:varprob2d} is an element of $H^1_0(\D) \cap H^2(\D)$.
\end{assumption}

We refer, for instance, to \cite[Theorem~3.2.1.2]{grisvard2011}, which gives
sufficient conditions for the existence of a strong solution. For example,
if $\D$ is a convex, bounded and open subset of $\R^2$ and if $\sigma \in
L^\infty(\D)$ has a globally Lipschitz continuous extension on $\overline{\D}$,
then Assumption~\ref{as:reg} is satisfied for every $f \in L^2(\D)$.

Next, we briefly review the finite element method for problem \eqref{eq:BVP}.
To this end, let $(\mathcal{T}_h)_{h \in (0,1]}$ be a family of
admissible triangulations of $\D$. More precisely, for every $h \in (0,1]$ it
holds that each triangle $T \in \mathcal{T}_h$ is an open subset of $\D$
satisfying
\begin{align*}
  \bigcup_{T \in \mathcal{T}_{h}} \overline{T} &= \overline{\D}
  \quad \text{ and } T \cap T' = \emptyset, \quad \text{for all } T,T' \in
  \mathcal{T}_h, T \neq T'.
\end{align*}
Further, it is assumed that no vertex of any triangle lies in the
interior of an edge of any other triangle of the triangulation, cf.
\cite[Definition~3.3.11]{brenner2008}. Typically, the parameter $h \in (0,1]$
denotes the maximal edge length of all triangles in $\mathcal{T}_h$.
Moreover, the area of a triangle $T$ is denoted by $|T|$.

As usual, we define the finite element space $S_h$ associated to a
triangulation $\mathcal{T}_h$, $h \in (0,1]$, by
\begin{align*}
  S_h = \{ v_h \in C(\overline{\D})\, : \, v_h = 0 \, \text{ on } \partial
  \D, \, v_h|_{T} \in \Pi_1 \, \forall 
  T \in \mathcal{T}_h\}. 
\end{align*}
Hereby, the set $\Pi_1$ consists of all polynomials up to degree $1$.
The finite element space $S_h$ is finite dimensional and $N_h = \dim(S_h)$
is called the \emph{number of degrees of freedom}. It coincides with the
number of interior nodes $(z_i)_{i = 1}^{N_h}$ 
of the triangulation. By $(\varphi_j)_{j = 1}^{N_h} \subset
S_h$ we denote the standard Lagrange basis of $S_h$ determined by
$\varphi_j(z_i) = \delta_{i,j}$ for all $i,j = 1,\ldots,N_h$. 
Further details on the construction of finite element spaces are found,
e.g., in \cite[Chapter~3]{brenner2008} or \cite[Chapter~5]{larsson2009}. 

For the error analysis in Section~\ref{sec:general} and Section
\ref{sec:importanceSampling} we have to impose the following additional
condition on the family of triangulations.

\begin{assumption}
  \label{as:triangulation}
  We assume that $(\mathcal{T}_h)_{h \in (0,1]}$ is a family of
  admissible and quasi-uniform triangulations. In particular, the interior
  angles of each triangle in $\mathcal{T}_h$ are bounded from below by a
  positive constant,   independently of $h$. In addition, there exists $c \in
  (0,\infty)$ such that for every $h \in (0,1]$ and $T \in \mathcal{T}_h$ it
  holds that $|T| \ge c h^2$. 
\end{assumption}

The assumption enables us to make use of a maximum norm estimate for functions
from the finite element space $S_h$, which we cite from
\cite[Lemma~6.4]{thomee2006}:
If Assumption~\ref{as:triangulation} is satisfied then there exists $C \in
(0,\infty)$, independently of $h \in (0,1]$, such that
\begin{align}
  \label{eq:maxnorm}
  \| v_h \|_{L^\infty(\D)} \le C \ell_h^{\frac{1}{2}} | v_h |_{H^1(\D)} 
\end{align}
for every $v_h \in S_h$, where $\ell_h = \max(1, \log(1/h))$.

Further, we recall that for a quasi-uniform family of triangulations 
the following inverse estimate is satisfied
\begin{align}
  \label{eq:inverse}
  | v_h |_{H^1(\D)} \le C h^{-1} \| v_h \|_{L^2(\D)}
\end{align}
for every $v_h \in S_h$, where $C$ is independent of the triangulation
$\mathcal{T}_h$. For a proof of \eqref{eq:inverse}
we refer to \cite[Section~4.5]{brenner2008}.

Next, we introduce the \emph{Ritz projector} $R_h \colon {H}^{1}_0(\D) \to S_h$
as the orthogonal projector onto $S_h$ with respect to the bilinear form $a$.
To be more precise, as a consequence of the lemma of Lax--Milgram,
for each $v \in {H}^{1}_0(\D)$ there exists a unique element
$R_h v \in S_h$ fulfilling
\begin{align}
  \label{eq:Ritz}
  a(R_h v, v_h) = a(v,v_h) \quad \text{ for all } v_h \in S_h.
\end{align}
Note that $R_h \colon H^1_0(\D) \to S_h$ is a  bounded linear operator.
In addition, there exists $C \in (0,\infty)$ such that for every $h \in (0,1]$
and $v \in H^1_0(\D) \cap H^2(\D)$ it holds
\begin{align}
  \label{eq:RitzH1}
  | (R_h - I) v|_{H^1(\D)} &\le C h \|v\|_{H^2(\D)},\\
  \label{eq:RitzL2}
  \| (R_h - I)v \|_{L^2(\D)} &\leq C h^2 \|v\|_{H^2(\D)}.
\end{align}
A proof is found, for instance, in \cite[Theorem~5.5]{larsson2009}.

For the introduction and the error analysis of Monte Carlo methods, we
also require some fundamental concepts from probability and stochastic
analysis. For a general introduction readers are referred to standard
monographs on this topic, for instance \cite{kallenberg2002, klenke2014}.  
For the measure theoretical background see also \cite{bauer2001, cohn2013}.

First, let us recall that a \emph{probability space}
$(\Omega,\mathcal{F},\P)$ consists of a measurable space $(\Omega,\mathcal{F})$
endowed with a finite 
measure $\P$ satisfying $\P(\Omega) = 1$. The value $\P(A) \in [0,1]$ is
interpreted as the \emph{probability} of the \emph{event} $A \in \F$.
A mapping $X \colon \Omega \to \R^d$, $d \in \N$,
is called a \emph{random variable} if $X$ is $\F /
\mathcal{B}(\R^d)$-measurable, where $\mathcal{B}(\R^d)$ denotes the
Borel-$\sigma$-algebra generated by the set of all open subsets of $\R^d$.
More precisely, it holds true that
\begin{align*}
  X^{-1}(B)= \big\{ \omega \in \Omega\, : \, X(\omega)\in B \big\} \in
  \mathcal{F} 
\end{align*}
for all $B \in \mathcal{B}(\R^d)$. Every random variable induces a probability
measure on its image space. In fact, the measure $\P_X \colon
\mathcal{B}(\R^d) \to [0,1]$ given by $\P_X(B)=\P(X^{-1}(B))$ for all
$B \in \mathcal{B}(\R^d)$ is a probability measure on the measurable space
$(\R^d, \mathcal{B}(\R^d))$. Usually, $\P_X$ is called the \emph{distribution}
of $X$. 

If the distribution $\P_X$ of $X$ is absolutely continuous with respect to
the Lebesgue measure, then there exists a measurable, non-negative mapping $g_X
\colon \R^d \to \R$ with 
\begin{align*}
  \P_X(B) =  \P( X^{-1}(B) ) = \int_B g_X(x) \diff{x}
\end{align*}
for every $B \in \mathcal{B}(\R^d)$. 
The mapping $g_X$ is called the \emph{probability density function} of $X$ and 
we write $X \sim g_X(x) \diff{x}$.

Next, let us recall that a random variable $X \colon \Omega \to \R^d$ is
called \emph{integrable} if $\int_{\Omega} |X(\omega)|
\diff{\P(\omega)}<\infty$. 
Then, the \emph{expectation} of $X$ is defined as
$$\E[X]:=\int_{\Omega}X(\omega)\diff{\P(\omega)} = \int_{\R^d} x
\diff{\P_{X}(x)}.$$ 
We say that $X$ is centered if $\E[X] = 0$.

Moreover, we write $X \in L^p(\Omega;\R^d)$ with $p \in [1,\infty)$ if 
$\int_{\Omega}|X(\omega)|^p \diff{\P(\omega)}<\infty$. If $d=1$, 
then we simply write $L^p(\Omega) := L^p(\Omega;\R)$.
In addition, the set
$L^p(\Omega;\R^d)$ becomes a Banach space if we identify all random
variables that only differ on a set of measure zero (i.e. probability zero)
and if we endow $L^p(\Omega;\R^d)$ with the norm
\begin{align*}
  \| X \|_{L^p(\Omega;\R^d)} = \big( \E \big[ |X|^p \big]
  \big)^{\frac{1}{p}} = \Big( \int_{\Omega} |X(\omega) |^p
  \diff{\P(\omega)} \Big)^{\frac{1}{p}}.
\end{align*}

In Section \ref{sec:general}, we frequently encounter a family of
$\mathcal{U}(T)$-distributed random variables $(Z_T)_{T \in \mathcal{T}_h}$. 
This means that for each $T \in \mathcal{T}$ the  mapping $Z_T \colon \Omega
\to \R^2$ is a random variable that is \emph{uniformly distributed} on the
triangle $T$. More precisely, the distribution
$\P_{Z_T}$ of $Z_T$ is given by $\P_{Z_T}(A) = \frac{|A \cap T|}{|T|}$ for
every $A \in \B(\R^2)$. Moreover, it follows from the 
transformation theorem that the expectation of $v \circ Z_T$ 
for an arbitrary function $v\in L^1(\D)$ is given by
$$\E[v(Z_T)]= \int_{T} v(z)\frac{1}{|T|} \diff{z} = \int_{\D} v(z)
\frac{1}{|T|} \one_T(z) \diff{z},$$
where the mapping $g_{Z_T}(z) = \frac{1}{|T|} \one_{T}(z)$, $z \in \D \subset
\R^2$,
is the probability density function of $Z_T$. 

Further, we say that a family of $\R^d$-valued random variables $(X_n)_{n \in
\N}$ is \emph{independent} if for any finite subset $M \subset \N$ and for
arbitrary events 
$(A_m)_{m \in M} \subset \mathcal{B}(\R^d)$ we have the multiplication rule
\begin{align*}
  \P \Big( \bigcap_{m \in M} \{ \omega \in \Omega\, : \, X_m(\omega) \in A_m \}
  \Big) = \prod_{m \in M} \P \big( \{ \omega \in \Omega\, : \, X_m(\omega) \in
  A_m \} \big).
\end{align*}
On the level of distributions this basically means
that the joint distribution of each finite subfamily $(X_m)_{m \in M}$ is equal
to the product measure of the single distributions. This directly implies
the multiplication rule for the expectation
\begin{align}
  \label{eq:prod_ind}
  \E\Big[ \prod_{m \in M} X_m \Big] = \prod_{m \in M} \E \big[ X_m \big],
\end{align}
provided $X_m$ is integrable for each $m \in M$.

Finally, let us mention that we often encounter random variables taking
values in a function space instead of $\R^d$. For instance, in
Theorem~\ref{thm:existence} we construct a random variable with values in
$S_h \subset H^1_0(\D)$. Since $S_h$ is finite dimensional all notions for
$\R^d$-valued random variables carry over to this case in a straight-forward
way. However, we often use the norm of the Bochner space
$L^p(\Omega;V)$ with either $V = H^1_0(\D)$ or $V=L^2(\D)$, which is given by 
\begin{align*}
  \| X \|_{L^p(\Omega;V)} = \big( \E \big[ \|X\|^p_V \big]
  \big)^{\frac{1}{p}} = \Big( \int_{\Omega}\|X(\omega)\|^p_V
  \diff{\P(\omega)}
  \Big)^{\frac{1}{p}}
\end{align*}
for $p \in [1,\infty)$. For an introduction to Bochner spaces we refer
to \cite[Appendix~E]{cohn2013}.



%
%

\section{A randomized quadrature formula on a triangulation}
\label{sec:general}

As already mentioned in the introduction, quadrature rules are often used
for the assembly of the matrix-vector system \eqref{eq:matvec}
associated to the finite element method for \eqref{eq:varprob2d}.
In this section, we introduce a randomized quadrature formula,
which is linked to the underlying triangulation $\mathcal{T}_h$ 
of the finite element space $S_h$. We discuss the well-posedness of the
resulting method and derive error estimates in a similar way as for
deterministic quadrature rules shown in \cite[Section~5.6]{larsson2009}. 

Let $\mathcal{T}_h$, $h \in (0,1]$, be an admissible triangulation of $\D$.
For a given $v \in L^1(\D)$, we consider the following Monte Carlo estimator
\begin{align}
  \label{eq:MC}
  Q_{MC}[v] := \sum_{T \in \mathcal{T}_h} |T| v( Z_T), 
\end{align}
where we sum over all triangles of the triangulation $\mathcal{T}_h$. Hereby, 
$(Z_T)_{T \in \mathcal{T}_h}$ denotes an independent family of random variables
such that for each triangle $T \in \mathcal{T}_h$ the random variable 
$Z_T$ is uniformly distributed on $T$, that is $Z_T \sim \mathcal{U}(T)$.
We discuss the simulation of $Z_T$ and the implementation of $Q_{MC}$ in
Subsection~\ref{sec:uniform}. 

Observe that the randomized quadrature rule is independent of the considered 
equivalence class of $v \in L^1(\D)$. If $v(x) = \tilde{v}(x)$ for almost every
$x \in \D$, then it follows that $Q_{MC}[v] = Q_{MC}[\tilde{v}]$ with
probability one.

\begin{lemma}
  \label{lem:QMC}
  Let $\mathcal{T}_h$ be an admissible triangulation with maximal edge length
  $h \in (0,1]$. Then, the random quadrature rule $Q_{MC}$ is unbiased, i.e.,
  for every $v \in L^1(\D)$ it holds
  \begin{align*}
    \E \big[ Q_{MC}[v] \big] = \int_\D v(x) \diff{x}.
  \end{align*}
  Moreover, if $v \in L^2(\D)$ then it holds that
  \begin{align*}
    \E \Big[  \Big| \int_\D v(x) \diff{x} - Q_{MC}[v] \Big|^2 \Big]
    \le \frac{\sqrt{3}}{2} h^2 \| v \|_{L^2(\D)}^2.
  \end{align*}
  In addition, if $v \in W^{s,2}(\D)$ for some $s \in (0,1)$ 
  then it follows that 
  \begin{align*}
    \E \Big[  \Big| \int_\D v(x) \diff{x} - Q_{MC}[v] \Big|^2 \Big]
    \le h^{2 + 2s} | v |_{W^{s,2}(\D)}^2.
  \end{align*}
\end{lemma}

\begin{proof}
  Due to $Z_T \sim \frac{1}{|T|} \one_T(z) \diff{z}$ for every $T \in
  \mathcal{T}_h$ we have
  \begin{align*}
    \E \big[ |T| v(Z_T) \big]
    = |T| \int_\D v(z) \frac{1}{|T|} \one_T(z) \diff{z}= \int_T v(z) \diff{z}.
  \end{align*}
  Then the first assertion follows by summing over all triangles of the
  triangulation.

  Now, let $v \in L^2(\D)$ be arbitrary. Then, the mean-square error is equal
  to 
  \begin{align*}
    \E \Big[ \Big|  \int_\D v(x) \diff{x} - Q_{MC}[v] \Big|^2 \Big]
    &= \E \Big[ \Big| \sum_{T \in \mathcal{T}_h} \Big( 
    \int_T v(x) \diff{x} - |T| v(Z_T) \Big) \Big|^2 \Big]\\
    &= \sum_{T \in \mathcal{T}_h} \E \Big[ \Big|
    \int_T v(x) \diff{x} - |T| v(Z_T) \Big|^2 \Big]
  \end{align*}
  since the summands are independent and centered random variables. 
  Therefore, they are orthogonal with respect to the $L^2(\Omega)$-inner
  product as can easily be deduced from \eqref{eq:prod_ind}.
  
  Next, for every $T \in \mathcal{T}_h$ we make use
  of $Z_T \sim \frac{1}{|T|} \one_T(z) \diff{z}$ and the Cauchy--Schwarz
  inequality. This yields
  \begin{align}
    \label{eq:err1}
    \E \Big[ \Big| \int_T v(x) \diff{x} - |T| v(Z_T) \Big|^2 \Big]
    &= |T|^2 \E \Big[ \Big| \frac{1}{|T|} \int_T v(x) \diff{x} - v(Z_T) \Big|^2
    \Big] \notag\\
    &= |T| \int_{T} \Big| \frac{1}{|T|} \int_T v(x) \diff{x} - v(z) \Big|^2
    \diff{z}\\
    &\le \int_{T} \int_{T} \big| v(x) - v(z) \big|^2 \diff{x} \diff{z}.
    \notag
  \end{align}
  Then, since $v \in L^2(\D)$ we get
  \begin{align*}
    \int_{T} \int_{T} \big| v(x) - v(z) \big|^2 \diff{x} \diff{z}
    &= \int_T \int_T \big( v(x)^2 - 2 v(x) v(z) + v(z)^2 \big) \diff{x}
    \diff{z}\\
    &= 2 |T| \int_{T} \big| v(x) \big|^2 \diff{x} - 
    2 \Big(\int_T v(x) \diff{x} \Big)^2\\
    &\le 2 |T| \int_{T} \big| v(x) \big|^2 \diff{x}. 
  \end{align*}
  Then, we recall Weitzenb\"ock's inequality \cite{weitzenboeck1919},
  which yields an upper bound for the area $|T|$ of a triangle $T \in
  \mathcal{T}_h$ with maximal edge length $h$. More precisely, it holds
  \begin{align}
    \label{eq:Weitzenb}
    |T| \le \frac{\sqrt{3}}{4} h^2.
  \end{align}
  Hence, after summing over all triangles we obtain
  \begin{align*}
    \Big\| \int_\D v(x) \diff{x} - Q_{MC}[v] \Big\|_{L^2(\Omega)}^2
    \le 2 \sum_{T \in \mathcal{T}_h} |T| \int_{T} \big| v(x) \big|^2 \diff{x}
    \le \frac{\sqrt{3}}{2} h^2 \| v \|^2_{L^2(\D)}.
  \end{align*}
  This proves the second claim.

  Finally, let $v \in W^{s,2}(\D)$, $s \in (0,1)$. 
  The estimate in \eqref{eq:err1} is then continued by
  \begin{align*}
    \E \Big[ \Big| \int_T v(x) \diff{x} - |T| v(Z_T) \Big|^2 \Big]
    &\le \int_{T} \int_{T} \big| v(x) - v(z) \big|^2 \diff{x} \diff{z}\\
    &\le h^{2 + 2s} \int_{T} \int_{T} \frac{ \big| v(x) - v(z)
    \big|^2}{|x - z|^{2 + 2 s}} \diff{x} \diff{z}\\
    &= h^{2 + 2s} | v |_{W^{s,2}(T)}^2
  \end{align*}
  since $|x - z| \le h$ for all $x,y \in T$. After summing over all
  triangles we directly obtain the third assertion. 
\end{proof}

Next, we apply the randomized quadrature formula \eqref{eq:MC} 
for the approximation of the bilinear form $a$ and the linear form $F$
defined in \eqref{eq:a_2d} and \eqref{eq:F_2d}. From this we obtain
two randomized mappings $a_{MC} \colon  S_h \times S_h  \to
L^\infty(\Omega)$ and $F_{MC} \colon S_h \to L^2(\Omega)$ which are given by 
\begin{align}
  \label{eq:aMC}
  a_{MC}(v_h,w_h) := Q_{MC}[\sigma \nabla v_h \cdot \nabla w_h] 
  = \sum_{T \in \mathcal{T}_h} |T| \sigma(Z_T) \nabla v_h(Z_T) \cdot \nabla
  w_h(Z_T) 
\end{align}
and
\begin{align}
  \label{eq:FMC}
  F_{MC}(v_h) := Q_{MC}[f v_h] = 
  \sum_{T \in \mathcal{T}_h} |T| f(Z_T) v_h(Z_T) 
\end{align}
for all $v_h, w_h \in S_h$. In passing, we observe that $a_{MC}(v_h,w_h) =
a(v_h,w_h)$ if $\sigma \equiv c \in ( 0,\infty)$ in $\D$. This holds true since
the gradients of $v_h, w_h \in S_h$ are constant on each triangle.

The next lemma answers the question of
well-posedness of $a_{MC}$ and $F_{MC}$ and contains some additional
properties.  

\begin{lemma}
  \label{lem:wellposedness}
  Let $(\mathcal{T}_h)_{h \in (0,1]}$ be a family of admissible triangulations
  of $\D$. Assume that $\sigma \in L^\infty(\D)$ satisfies $\sigma(x) \ge
  \sigma_0 > 0$ for almost every $x \in \D$. Then, the 
  mapping $a_{MC}$ introduced in \eqref{eq:aMC} is well-defined for every $h
  \in (0,1]$. Moreover, it holds $\P$-almost surely that
  \begin{align*}
    |a_{MC}(v_h,w_h)| &\le \| \sigma \|_{L^\infty(\D)} | v_h |_{H^1(\D)}
    |w_h|_{H^1(\D)},\\
    a_{MC}(v_h, v_h) &\ge \sigma_0 |v_h|_{H^1(\D)}^2  
  \end{align*}
  for all $v_h, w_h \in S_h$.

  In addition, if $f \in L^2(\D)$ and the family of triangulations satisfies 
  Assumption~\ref{as:triangulation} then the mapping $F_{MC}$ defined in
  \eqref{eq:FMC} is also well-defined and there exists $C \in (0,\infty)$
  independent of $\mathcal{T}_h$ with
  \begin{align*}
    |F_{MC}(v_h)| &\le C \ell_h^{\frac{1}{2}}
    Q_{MC}[|f|] | v_h |_{H^1(\D)} < \infty \quad \text{$\P$-a.s.},\\
    \| F_{MC}(v_h) \|_{L^2(\Omega)} &\le C \|f \|_{L^2(\D)} |v_h|_{H^1(\D)}    
  \end{align*}
  for all $v_h \in S_h$, where $\ell_h = \max( 1 , \log(1/h))$.
\end{lemma}

\begin{proof}
  We first show that $a_{MC}(v_h, w_h) \in L^\infty(\Omega)$ for every $v_h,
  w_h \in S_h$. 
  To see this, we recall that the functions in $S_h$ are linear on
  each triangle $T$ in $\mathcal{T}_h$. This implies that
  the gradient $\nabla v_h$ is piecewise constant for every $v_h \in S_h$.
  Hence, the random variables $\nabla v_h(Z_T)$, $T \in \mathcal{T}_h$,
  are, in fact, constant with
  probability one. This implies that
  \begin{align*}
    |T| |\nabla v_h(Z_T)|^2 = \int_T |\nabla v_h(x)|^2 \diff{x}\quad
    \text{$\P$-almost surely.}
  \end{align*}
  Together with the assumption $\sigma \in L^\infty(\D)$ it therefore follows
  that the summands in \eqref{eq:aMC} are essentially bounded
  random variables. More 
  precisely, it holds $\P$-almost surely that
  \begin{align*}
    |a_{MC}(v_h, w_h)| &\le \sum_{T \in \mathcal{T}_h} |T| \sigma(Z_T) 
    | \nabla v_h(Z_T)| |\nabla w_h(Z_T)| \\
    &\le \|\sigma\|_{L^\infty(\D)}  
    \sum_{T \in \mathcal{T}_h} |T| | \nabla v_h(Z_T)| |\nabla w_h(Z_T)|\\
    &\le \|\sigma\|_{L^\infty(\D)} \Big( \sum_{T \in \mathcal{T}_h} |T|
    | \nabla v_h(Z_T)|^2 \Big)^{\frac{1}{2}}
    \Big( \sum_{T \in \mathcal{T}_h} |T|
    | \nabla w_h(Z_T)|^2 \Big)^{\frac{1}{2}}\\
    &= \|\sigma\|_{L^\infty(\D)} |v_h|_{H^1(\D)} |w_h|_{H^1(\D)}
  \end{align*}
  for all $v_h, w_h \in S_h$.

  Moreover, the same arguments yield for every $v_h \in S_h$
  \begin{align*}
    a_{MC}(v_h, v_h) = \sum_{T \in \mathcal{T}_h} |T| \sigma(Z_T) 
    | \nabla v_h(Z_T)|^2 \ge \sigma_0 |v_h|^2_{H^1(\D)}
    \quad \text{$\P$-almost surely,}
  \end{align*}
  since $\sigma(Z_T) \ge \sigma_0 > 0$ almost surely.

  Next, we turn to the mapping $F_{MC}$. 
  From \eqref{eq:maxnorm} it follows for $v_h \in S_h$ that
  \begin{align*}
    |F_{MC}(v_h)| &\le \sum_{T \in \mathcal{T}_h} |T| |f(Z_T)| |v_h(Z_T)|
    \le \| v_h \|_{L^\infty(\D)} Q_{MC}[|f|]\\
    &\le C \ell_h^{\frac{1}{2}}
    Q_{MC}[|f|] | v_h |_{H^1(\D)}.
  \end{align*}
  Observe that the bound on the right-hand side still contains a random
  quadrature formula and is, therefore, itself random. However,
  for $f \in L^2(\D)$ it follows from applications of the Cauchy--Schwarz
  inequality and Lemma~\ref{lem:QMC} that
  \begin{align*}
    \E \big[ ( Q_{MC}[|f|])^2 \big]
    &= \E \Big[ \Big( \sum_{T \in \mathcal{T}_h} |T| |f(Z_T)| \Big)^2 \Big]\\
    &\le |\D| \E \Big[ \sum_{T \in \mathcal{T}_h} |T| |f(Z_T)|^2  \Big]\\
    &= |\D| \int_\D |f(z)|^2 \diff{z}.
  \end{align*}
  In particular, we have that $Q_{MC}[|f|] < \infty$ with probability one.
  This also proves that $F_{MC}(v_h) \in L^2(\Omega)$. It remains to prove the
  asserted estimate of the $L^2(\Omega)$-norm of $F_{MC}(v_h)$. For this we
  first observe that
  \begin{align*}
    \| F_{MC}(v_h) \|^2_{L^2(\Omega)} = 
    \big\| F_{MC}(v_h) - \E [ F_{MC}(v_h) ] \big\|^2_{L^2(\Omega)}
    + \big( \E [ F_{MC}(v_h) ] \big)^2
  \end{align*}
  for every $v_h \in S_h$. 
  From Lemma~\ref{lem:QMC}, the Cauchy--Schwarz inequality, and the Poincar\'e
  inequality on $H^1_0(\D)$ it follows that
  \begin{align*}
    \big( \E [ F_{MC}(v_h)] \big)^2 &= \Big( \int_{\D} f(x) v_h(x) \diff{x}
    \Big)^2\\
    &\le \int_{\D} |f(x)|^2 \diff{x} \int_{\D} |v_h(x)|^2 \diff{x} \le C \| f
    \|_{L^2(\D)}^2 | v_h |_{H^1(\D)}^2,
  \end{align*}
  where the constant $C$ only depends on $\D$.
  An application of Lemma~\ref{lem:QMC} then yields
  \begin{align*}
    \E \big[ \big| F_{MC}(v_h) - \E [ F_{MC}(v_h) ] \big|^2 \big]
    &= \E \Big[ \Big| Q_{MC}[f v_h] - \int_\D f(x) v_h(x) \diff{x} \Big|^2
    \Big]\\
    &\le \frac{\sqrt{3}}{2} h^2 \| f v_h \|_{L^2(\D)}^2\\
    &\le \frac{\sqrt{3}}{2} h^2 \|f \|_{L^2(\D)}^2 \|v_h\|_{L^\infty(\D)}^2\\
    &\le C \frac{\sqrt{3}}{2} h^2 \ell_h \|f \|_{L^2(\D)}^2
    |v_h|_{H^1(\Omega)}^2,
  \end{align*}
  where we also applied the maximum norm estimate \eqref{eq:maxnorm}.
  Hence, after taking note of $\sup_{h \in (0,1]} h^2 \ell_h =
  \sup_{h \in (0,1]} h^2 \max(1,\log(1/h)) < \infty$ the proof is completed.
\end{proof}

Next, we introduce the finite element problem based on the randomized
quadrature rule. In terms of $a_{MC}$ and $F_{MC}$ the problem is stated as
follows:
\begin{align}
  \label{eq:discProb}
  \begin{cases}
    \text{Find } u_h^{MC} \colon \Omega \to S_h \text{ such that $\P$-almost
    surely} \\
    a_{MC}(u_h^{MC},v_h) = F_{MC}(v_h) \text{ for all } v_h \in S_h.
  \end{cases}
\end{align}

\begin{theorem}
  \label{thm:existence}
  Suppose that $f \in L^2(\D)$ and $\sigma \in L^\infty(\D)$ with $\sigma(x)
  \ge \sigma_0 > 0$ for almost every $x \in \D$ are given. Then,
  for every admissible triangulation $\mathcal{T}_h$, $h \in (0,1]$, there
  exists a uniquely determined solution $u_h^{MC} \colon \Omega \to
  S_h$ to the discrete problem \eqref{eq:discProb}. In addition,
  there exists $C \in (0,\infty)$ independent of $\mathcal{T}_h$ such that
  \begin{align*}
    |u_h^{MC}|_{H^1(\D)} 
    &\le C \ell_h^{\frac{1}{2}} Q_{MC}[|f|] \quad \text{$\P$-a.s.,}
  \end{align*}
  where $\ell_h = \max(1,\log(1/h))$.
\end{theorem}

\begin{proof}
  It follows from Lemma~\ref{lem:wellposedness} that the bilinear form
  $a_{MC}$ is $\P$-almost surely strictly positive and bounded. Moreover, an
  inspection of the proof reveals that the exceptional set $N_1 \subset \Omega$
  of probability zero, where these properties might be violated, can be chosen
  independently of the arguments $v_h, w_h \in S_h$. This is true since
  only the gradients of $v_h$ and $w_h$ appear in $a_{MC}(v_h,w_h)$, which are
  piecewise constant on each triangle. Hence, on the set $\{ Z_T \in T\} \in
  \F$, which has probability one, the randomness only occurs in the coefficient
  function $\sigma$.  
  Therefore, for every $\omega \in \Omega \setminus N_1$ the mapping
  $S_h \times S_h \ni (v_h,w_h) \mapsto a_{MC}(v_h,w_h)(\omega) \in \R$
  satisfies the conditions of the lemma of Lax--Milgram.

  In the same way, there exists a measurable set $N_2 \subset \Omega$ of
  probability zero such that the mapping $S_h \ni v_h \mapsto F_{MC}(v_h)(\omega)
  \in \R$ is a bounded linear functional on $H^1_0(\D)$ for all
  $\omega \in \Omega\setminus N_2$. In particular, we observe
  that the exceptional set $N_2$ can again be chosen independently
  of the mapping $v_h$ due to the continuity of all elements in $S_h$. In
  addition, the following estimate, which was used in the proof of
  Lemma~\ref{lem:wellposedness}, is true for all $\omega \in \Omega$:
  \begin{align*}
    |v_h(Z_T(\omega))| \le \| v_h \|_{L^\infty(\D)}.
  \end{align*}
  Consequently, for every fixed $\omega \in \Omega \setminus (N_1 \cup N_2)$
  the lemma of Lax--Milgram uniquely determines an element $u_h^{MC}(\omega) \in
  S_h$ satisfying 
  \begin{align}
    \label{eq:discProb2}
    a_{MC}(u_h^{MC}(\omega), v_h)(\omega) = F_{MC}(v_h)(\omega)
    \quad \text{for all } v_h \in S_h.
  \end{align}
  Let us define $u_h^{MC}(\omega) = 0 \in S_h$ for all $\omega \in N_1 \cup N_2$. 
  Next, we have to prove that the mapping $\Omega \ni \omega \mapsto
  u_h^{MC}(\omega) \in S_h$ is measurable. However, this follows from an
  application of \cite[Lemma~4.3]{eisenmann2017} to the mapping $g \colon
  \Omega \times \R^{N_h} \to \R^{N_h}$ defined by $g(v ,\omega) := [
  a_{MC}( \sum_{i = 1}^{N_h} v_i \psi_i, \psi_j)(\omega) -
  F_{MC}(\psi_j)]_{j = 1}^{N_h}$, where $v = [v_i]_{i = 1}^{N_h} \in
  \R^{N_h}$ and $(\psi_j)_{j = 1}^{N_h} \subset S_h$ is an arbitrary basis
  of the finite dimensional space $S_h$.
  
  It remains to prove the stability estimate. Due to
  Lemma~\ref{lem:wellposedness} and \eqref{eq:discProb2}
  it holds on $\Omega \setminus (N_1 \cup N_2)$ that
  \begin{align*}
    \sigma_0 |u_h^{MC}|^2_{H^1(\D)} 
    \le a_{MC}(u_h^{MC},u_h^{MC}) 
    = F_{MC}(u_h^{MC})
    \le C \ell_h^{\frac{1}{2}} Q_{MC}[|f|] |u_h^{MC}|_{H^1(\D)}.
  \end{align*}
  Hence, after canceling the norm of $u_h^{MC}$ one time on both sides of the
  inequality we obtain the desired estimate.
\end{proof}

Let us emphasize that the solution to the discrete problem \eqref{eq:discProb}
is a random variable. In fact, it follows directly from
Theorem~\ref{thm:existence} that $u_h^{MC} \in L^p(\Omega;H^1_0(\D))$ provided
$f \in L^p(\D)$ for $p \in [2,\infty]$. 

As in the standard error analysis
(cf. \cite[Theorem~5.7]{larsson2009}), we want to use 
$u_h^{MC}$ as a test function in the discrete problem \eqref{eq:discProb}.
However, in contrast to the situation in Lemma~\ref{lem:QMC} we 
have, in general, that $|\E[ F_{MC}(v_h)] -  F( \E[v_h] )| \neq 0$ for an
arbitrary $S_h$-valued random function $v_h \in L^2(\Omega;H^1_0(\D))$. 
The following lemma gives an estimate of this difference. 

\begin{lemma}
  \label{lem:FMC}
  Let Assumption~\ref{as:triangulation} be satisfied.
  Then, there exists $C \in (0,\infty)$ such that
  for every $h \in (0,1]$, $f \in L^p(\D)$, $p \in [2,\infty]$, 
  and $S_h$-valued random variable $v_h \in L^2(\Omega;H^1_0(\D))$ 
  it holds
  \begin{align*}
    \big| \E \big[ F_{MC}(v_h) - F(v_h) \big] \big| 
    \le 
    \begin{cases}
      C h^{1 - \frac{2}{p}} 
      \|f\|_{L^p(\D)} \| v_h \|_{L^2(\Omega;H^1_0(\D))},
      &\quad \text{ if } p \in [2,\infty),\\
      C \ell_h^{\frac{1}{2}} h \|f\|_{L^\infty(\D)} \| v_h
      \|_{L^2(\Omega;H^1_0(\D))}, 
      &\quad \text{ if } p = \infty,
    \end{cases}
  \end{align*}
  where $\ell_h = \max(1,\log(1/h))$. 
\end{lemma}

\begin{proof}
  For the error analysis it is convenient to
  choose an $L^2(\D)$-orthonormal basis $(\psi_j)_{j = 1}^{N_h}$ of $S_h$, 
  which solves the discrete eigenvalue problem
  \begin{align}
    \label{eq:eigval}
    a(\psi_j, w_h) = \lambda_{h,j} (\psi_j, w_h)_{L^2(\D)} 
  \end{align}
  for all $w_h \in S_h$. Hereby, $0 < \lambda_{h,1} \le \lambda_{h,2} \le 
  \ldots \le \lambda_{h,N_h}$ denote the discrete eigenvalues of the bilinear
  form $a$ on the finite element space $S_h \subset H^{1}_0(\D)$. We refer to
  \cite[Section~6.2]{larsson2009} regarding the existence of
  $(\lambda_{h,j})_{j = 1}^{N_h}$ and the associated 
  orthonormal basis $(\psi_j)_{j = 1}^{N_h}$. 

  Next, let $h \in (0,1]$, $f \in L^p(\D)$, $p \in [2,\infty]$, and an
  $S_h$-valued random variable $v_h \in L^2(\Omega;H^1_0(\D))$ be arbitrary.
  Then, we represent $v_h$ in terms of the orthonormal
  basis $(\psi_j)_{j = 1}^{N_h} \subset S_h$ by
  \begin{align}
    \label{eq:v_h_rep}
    v_h = \sum_{j = 1}^{N_h} v_j \psi_j,
  \end{align}
  For this choice of the basis, the random coefficients $(v_j)_{j = 1}^{N_h}
  \subset L^2(\Omega)$ are given by   
  \begin{align*}
    v_j = ( v_h, \psi_j)_{L^2(\D)}.
  \end{align*}
  In particular, it follows from the Cauchy--Schwarz inequality that $v_j$ is
  indeed a real-valued and square-integrable random variable for every $j \in
  \{1,\ldots,N_h\}$.  
  Due to the linearity of $F$ and $F_{MC}$ we then arrive at the estimate 
    \begin{align*}
    \big| \E \big[ F_{MC}(v_h) - F(v_h) \big] \big| 
    &= \Big| \sum_{j = 1}^{N_h} \E\big[ v_j \big( F_{MC}(\psi_j) -
    F(\psi_j)\big) \big] \Big|\\
    &\le \sum_{j = 1}^{N_h} \big( \E \big[ |v_j |^2 \big]
    \big)^{\frac{1}{2}} \big( \E \big[ \big|  F_{MC}(\psi_j) -
    F(\psi_j)\big|^2 \big] \big)^{\frac{1}{2}}\\
    &\le \Big( \sum_{j = 1}^{N_h} \lambda_{h,j} \E \big[ |v_j |^2 \big]
    \Big)^{\frac{1}{2}}
    \Big( \sum_{j = 1}^{N_h} \lambda_{h,j}^{-1}
    \E \big[ \big|  F_{MC}(\psi_j) - F(\psi_j)\big|^2 \big]
    \Big)^{\frac{1}{2}}
  \end{align*}
  by additional applications of the Cauchy--Schwarz inequality.
  From \eqref{eq:v_h_rep} and \eqref{eq:eigval} we then get
  \begin{align*}
    a (v_h, v_h)
    = \sum_{i,j=1}^{N_h} v_j v_i a(\psi_j,\psi_i)
    = \sum_{i,j=1}^{N_h} \lambda_{h,j} v_i v_j
    (\psi_j,\psi_i)_{L^2(\D)}
    = \sum_{j=1}^{N_h} \lambda_{h,j} v_j^2,
  \end{align*}
  since $(\psi_{j})_{j = 1}^{N_h}$ is an orthonormal basis of $S_h$.
  From this it follows that
  \begin{align}
    \label{eq:v_hH1}
    \Big( \sum_{j = 1}^{N_h} \lambda_{h,j} \E \big[ |v_j |^2 \big]
    \Big)^{\frac{1}{2}}
    = \big( \E \big[ a(v_h,v_h) \big] \big)^\frac{1}{2}
    \le \|\sigma\|_{L^\infty(\D)}^{\frac{1}{2}}
    \| v_h \|_{L^2(\Omega;H^1_0(\D))}.
  \end{align}
  Moreover, an application of Lemma~\ref{lem:QMC} yields
  \begin{align*}
    \E \big[ \big|  F_{MC}(\psi_j) - F(\psi_j)\big|^2 \big]
    = \E \Big[ \Big|  Q_{MC}(f \psi_j) - \int_\D f \psi_j \diff{x}
    \Big|^2 \Big] 
    \le \frac{\sqrt{3}}{2} h^2 \| f \psi_j\|_{L^2(\D)}^2
  \end{align*}
  for every $j \in \{1,\ldots,N_h\}$. 
  Further, since $f \in L^p(\D)$, $p \in [2,\infty]$, it follows from an
  application of H\"older's inequality with conjugated exponent $p' \in
  [2,\infty]$ determined by $\frac{1}{p} + \frac{1}{p'} = \frac{1}{2}$ that
  \begin{align*}
    \| f \psi_j\|_{L^2(\D)}
    \le \| f \|_{L^p(\D)}  \| \psi_j\|_{L^{p'}(\D)}.
  \end{align*}
  An application of the Gagliardo--Nierenberg inequality, 
  cf.~\cite[Theorem~1.24]{roubicek2013}, yields
  \begin{align*}
    \| \psi_j \|_{L^{p'}(\D)} \le C \| \psi_j
    \|_{L^2(\D)}^{\frac{2}{p'}} | \psi_j|_{H^1(\D)}^{1 - \frac{2}{p'}},
  \end{align*}
  where the constant $C$ is independent of $j \in \{1,\ldots,N_h\}$.
  Since $\| \psi_j \|_{L^2(\D)} = 1$ for every $j \in
  \{1,\ldots,N_h\}$ and due to \eqref{eq:a_pos} and \eqref{eq:eigval} we
  therefore obtain 
  \begin{align*}
    \| \psi_j \|_{L^{p'}(\D)} \le C | \psi_j|_{H^1(\D)}^{\frac{2}{p}}
    \le \frac{C}{\sigma_0^{\frac{1}{p}}} a(\psi_j,\psi_j)^{\frac{1}{p}} 
    \le \frac{C}{\sigma_0^{\frac{1}{p}}} \lambda_{h,j}^{\frac{1}{p}} 
  \end{align*}
  for every $p, p' \in [2,\infty]$ with $\frac{1}{p} + \frac{1}{p'} =
  \frac{1}{2}$. Altogether, we have the bound
  \begin{align}
    \label{eq:step2}
    \begin{split}
      \sum_{j = 1}^{N_h} \lambda_{h,j}^{-1}
      \E \big[ \big|  F_{MC}(\psi_j) - F(\psi_j)\big|^2 \big]
      &\le \frac{\sqrt{3}}{2} h^2 \| f \|_{L^p(\D)}^2 \sum_{j = 1}^{N_h}
      \lambda_{h,j}^{-1} \| \psi_j \|_{L^{p'}(\D)}^2 \\ 
      &\le C h^{2}
      \| f \|_{L^p(\D)}^2 \sum_{j = 1}^{N_h}
      \lambda_{h,j}^{-1 + \frac{2}{p}}.
    \end{split}
  \end{align}
  Concerning the last sum we recall from \cite[Theorem~6.7]{larsson2009} that
  \begin{align*}
    \lambda_{j} \le \lambda_{h,j} 
  \end{align*}
  for all $j \in \{1,\ldots,N_h\}$, where $(\lambda_j)_{j \in \N}$ denotes the
  family of eigenvalues of the bilinear form $a$ on the full space $H^1_0(\D)$.
  Moreover, it is well-known, cf.~\cite[Section~6.1]{larsson2009}, that 
  there exist constants $c_1, c_2 \in (0,\infty)$ only depending on $\sigma$
  and $\D$ such that 
  \begin{align*}
    c_1 j \le \lambda_j \le c_2 j.
  \end{align*}
  From this it follows that
  \begin{align*}
    \sum_{j = 1}^{N_h} \lambda_{h,j}^{-1+ \frac{2}{p}}
    &\le \sum_{j = 1}^{N_h} \lambda_{j}^{-1+ \frac{2}{p}}
    \le c_1^{-1 + \frac{2}{p}} \sum_{j = 1}^{N_h} j^{-1+ \frac{2}{p}}
    \le c_1^{-1 + \frac{2}{p}}
    \Big(1 + \int_{1}^{N_h} y^{-1 + \frac{2}{p}} \diff{y}
    \Big).
  \end{align*}
  Hence, we obtain
  \begin{align*}
    \sum_{j = 1}^{N_h} \lambda_{h,j}^{-1+ \frac{2}{p}}
    \le 
    \begin{cases}
      \frac{p}{2} c_1^{-1 + \frac{2}{p}} N_h^{\frac{2}{p}},
      &\quad \text{if } p \in [2,\infty),\\
      c_1^{-1} (1 + \log(N_h)),
      &\quad \text{if } p = \infty.
    \end{cases}
  \end{align*}
  From \eqref{eq:eigval}, \eqref{eq:a_bdd}, and 
  the inverse estimate \eqref{eq:inverse} it then follows that
  \begin{align*}
    N_h \le \frac{1}{c_1} \lambda_{h,N_h} 
    = \frac{1}{c_1} a(\psi_{N_h},\psi_{N_h})
    \le \frac{1}{c_1} \| \sigma \|_{L^\infty(\D)} | \psi_{N_h}|^2_{H^1(\D)}
    \le C h^{-2}.
  \end{align*}
  This implies that
  $\log(N_h) \le C \max(1, \log(1/h)) = C \ell_h$.
  Altogether, this yields
  \begin{align}
    \label{eq:sumlam-1}
    \sum_{j = 1}^{N_h} \lambda_{h,j}^{-1 + \frac{2}{p}}
    \le 
    \begin{cases}
      C h^{- \frac{4}{p}},&\quad \text{if } p \in [2,\infty),\\
      C \ell_h,&\quad \text{if } p = \infty.
    \end{cases}
  \end{align}
  Combining this with \eqref{eq:v_hH1} and \eqref{eq:step2} then completes the
  proof. 
\end{proof}

Next, we state and prove the main result of this section.

\begin{theorem}
  \label{thm:errorH1}
  Suppose that  $\sigma \in
  L^\infty(\D) \cap W^{s,q}(\D)$, $s \in (0,1]$, $q \in (2,\infty)$, with
  $\sigma(x) \ge \sigma_0 > 0$ for almost every $x \in\D$. 
  Let Assumptions~\ref{as:reg} and \ref{as:triangulation} be satisfied.
  If $f \in L^p(\D)$, $p \in [2,\infty)$, then it holds 
  \begin{align*}
    \big\| u - u_h^{MC} \big\|_{L^2(\Omega;H^1_0(\D))} 
    &\le C h \|u \|_{H^2(\D)} + C h^s \|u \|_{H^2(\D)} | \sigma |_{W^{s,q}(\D)}
    + C h^{1 - \frac{2}{p}} \| f \|_{L^p(\D)}
  \end{align*}
  for every $h \in (0,1]$.
  Further, if $f \in L^\infty(\D)$ then it holds
  \begin{align*}
    \big\| u - u_h^{MC} \big\|_{L^2(\Omega;H^1_0(\D))} 
    &\le C h \|u \|_{H^2(\D)} + C h^s \|u \|_{H^2(\D)} | \sigma |_{W^{s,q}(\D)}
    + C \ell_h^{\frac{1}{2}} h \| f \|_{L^\infty(\D)}
  \end{align*}
  for every $h \in (0,1]$.
\end{theorem}

\begin{proof}
  Let us split the error into the following two parts
  \begin{align*}
    u_h^{MC} - u = u_h^{MC} - R_h u + R_h u - u =: \theta + \rho,
  \end{align*}
  where $R_h \colon H^1_0(\D) \to S_h$ denotes the Ritz projector
  (see Section~\ref{sec:notation}). Observe that $\theta$ and $\rho$ are
  orthogonal with respect to the bilinear form $a$. 
  Then, it follows from the positivity \eqref{eq:a_pos} and
  boundedness \eqref{eq:a_bdd} of $a$ that
  \begin{align*}
    \sigma_0 |u_h^{MC} - u|^2_{H^1(\D)} 
    &\le a( u_h^{MC} - u, u_h^{MC} - u) = a( \theta, \theta) + a(\rho, \rho)\\
    &\le \| \sigma \|_{L^\infty(\D)} \big( |\theta|^2_{H^1(\D)} +
    |\rho|^2_{H^1(\D)} \big).
  \end{align*}
  Standard error estimates for the conforming finite element method,
  cf.~\eqref{eq:RitzH1}, yield
  \begin{align}
    \label{eq:rho}
    | \rho |_{H^1(\D)} = | R_h u - u|_{H^1(\D)} \le C h \|u\|_{H^2(\D)}.
  \end{align}
  Moreover, from \eqref{eq:varprob2d} and \eqref{eq:discProb} we get
  $\P$-almost surely for every $v_h \in S_h$ that
  \begin{align*}
    a_{MC}(\theta, v_h) &= a_{MC}(u_h^{MC}, v_h) - a_{MC}(R_h u, v_h)\\
    &= F_{MC}(v_h) - F(v_h) + a(R_h u,v_h) - a_{MC}(R_h u,v_h),
  \end{align*}
  since $F(v_h) = a(u,v_h) = a(R_h u, v_h)$ for every $v_h \in S_h$.
  In particular, for the choice $v_h = \theta(\omega) = u_h^{MC}(\omega) - R_h u \in
  S_h$ we obtain $\P$-almost surely that
  \begin{align*}
    \sigma_0 |\theta|^2_{H^1(\D)} 
    \le a_{MC}(\theta, \theta)
    = F_{MC}(\theta) - F(\theta) + a(R_h u, \theta)
    - a_{MC}(R_h u,\theta).
  \end{align*}
  From Lemma~\ref{lem:wellposedness} and Theorem~\ref{thm:existence} it follows
  directly that all terms on the right-hand side are integrable with respect to
  $\P$. Hence, after taking expectations it remains to give error estimates for
  the two terms
  \begin{align*}
    E_1 &= \big| \E\big[ F_{MC}(\theta) - F(\theta) \big] \big|,\\
    E_2 &= \big| \E\big[ a(R_h u, \theta) - a_{MC}(R_h u, \theta) \big] \big|.
  \end{align*}
  An application of Lemma~\ref{lem:FMC} directly yields 
  \begin{align*}
    E_1 \le 
    \begin{cases}
      C h^{1 - \frac{2}{p}} \| f \|_{L^p(\D)} \| \theta
      \|_{L^2(\Omega;H^1_0(\D))},&\quad \text{if } p \in [2,\infty),\\
      C \ell_h^{\frac{1}{2}} h \| f \|_{L^\infty(\D)}
      \| \theta \|_{L^2(\Omega;H^1_0(\D))},&\quad \text{if } p = \infty.  
    \end{cases}
  \end{align*}
  Next, we turn to the term $E_2$ which is given by 
  \begin{align*}
    E_2 &= \big| \E \big[ a(R_h u,\theta) - a_{MC}(R_h u, \theta) \big] \big|\\
    &= \Big| \sum_{T \in \mathcal{T}_h} \E \Big[ \int_T \sigma(x) \nabla R_h
    u(x) \cdot \nabla \theta(x) \diff{x} - |T| \sigma(Z_T) \nabla R_h u(Z_T)
    \cdot \nabla \theta (Z_T) \Big] \Big|.
  \end{align*}
  Since $R_h u \in S_h$ and $\theta \colon \Omega \to S_h$, the respective 
  gradients are constant on each triangle. Therefore, we have
  $\nabla R_h u (x) \cdot \nabla \theta (x) = \nabla R_h u (Z_T) \cdot \nabla
  \theta (Z_T)$ for every $x \in T$. Hence, we get
  \begin{align*}
    E_2 &= \Big| \sum_{T \in \mathcal{T}_h} \E \Big[ \Big( \int_T \sigma(x)
    \diff{x} - |T| \sigma(Z_T) \Big) \nabla R_h u(Z_T)
    \cdot \nabla \theta (Z_T) \Big] \Big|\\
    &\le \sum_{T \in \mathcal{T}_h} \Big( 
    \E \Big[ \Big| \Big(\int_T \sigma(x)
    \diff{x} - |T| \sigma(Z_T) \Big) \nabla R_h u(Z_T) \Big|^2 \Big]
    \Big)^{\frac{1}{2}} \big( \E \big[ | \nabla \theta(Z_T)|^2 \big]
    \big)^{\frac{1}{2}}\\
    &\le    
    \Big( \sum_{T \in \mathcal{T}_h} |T|^{-1} \E \Big[ 
    \Big( \int_T \sigma(x) \diff{x} - |T| \sigma(Z_T) \Big)^2
    \big| \nabla R_h u(Z_T)\big|^2 \Big] 
    \Big)^{\frac{1}{2}}\\
    &\quad \times \Big( \sum_{T \in \mathcal{T}_h} |T|
    \E\big[ | \nabla \theta (Z_T) |^2 \big]
    \Big)^{\frac{1}{2}}
  \end{align*}
  by further applications of the Cauchy--Schwarz inequality.
  Moreover, by making again use of the fact that the gradient of $\theta$
  is piecewise constant we obtain
  \begin{align*}
    \Big( \sum_{T \in \mathcal{T}_h} |T|
    \E\big[ | \nabla \theta (Z_T) |^2 \big]
    \Big)^{\frac{1}{2}}
    & =
    \Big( \E \Big[ \sum_{T \in \mathcal{T}_h} |T|
    | \nabla \theta (Z_T) |^2 \Big]
    \Big)^{\frac{1}{2}}\\
    &= \Big( \E \Big[ \int_\D |\nabla \theta(x)|^2
    \diff{x} \Big] \Big)^{\frac{1}{2}}
    = \| \theta\|_{L^2(\Omega;H^1_0(\D))}.
  \end{align*}
  Further, due to $Z_T \sim |T|^{-1} \one_{T}(z) \diff{z}$ it holds
  \begin{align}
    \label{eqn:est1}
    \begin{split}
    &\Big( \sum_{T \in \mathcal{T}_h} |T|^{-1} \E \Big[ 
    \Big( \int_T \sigma(x) \diff{x} - |T| \sigma(Z_T) \Big)^2
    | \nabla R_h u(Z_T)|^2 \Big] 
    \Big)^{\frac{1}{2}}\\
    &\quad = \Big( \sum_{T \in \mathcal{T}_h} 
    |T|^{-2} \int_T \Big( \int_T \big( \sigma(x) - \sigma(z) \big) \diff{x}
    \Big)^2 | \nabla R_h u(z) |^2 \diff{z} \Big)^{\frac{1}{2}}\\
    &\quad \le \Big( \sum_{T \in \mathcal{T}_h} 
    |T|^{-2} \int_T \Big( \int_T \big( \sigma(x) - \sigma(z) \big) \diff{x}
    \Big)^2 | \nabla (R_h - I) u(z) |^2 \diff{z} \Big)^{\frac{1}{2}}\\
    &\qquad + \Big( \sum_{T \in \mathcal{T}_h} 
    |T|^{-2} \int_T \Big( \int_T \big( \sigma(x) - \sigma(z) \big) \diff{x}
    \Big)^2 | \nabla u(z) |^2 \diff{z} \Big)^{\frac{1}{2}},
    \end{split}
  \end{align}
  where we applied Minkowski's inequality in the last step.
  The first term is then estimated by
  \begin{align*}
    &\Big( \sum_{T \in \mathcal{T}_h} 
    |T|^{-2} \int_T \Big( 
    \int_T \big( \sigma(x) - \sigma(z) \big) \diff{x} \Big)^2
    | \nabla (R_h - I) u(z) |^2 \diff{z} \Big)^{\frac{1}{2}}\\
    &\quad \le \Big( \sum_{T \in \mathcal{T}_h} 
    |T|^{-1} \int_T \int_T \big( \sigma(x) - \sigma(z) \big)^2 \diff{x} 
    | \nabla (R_h - I) u(z) |^2 \diff{z} \Big)^{\frac{1}{2}}\\
    &\quad \le C \| \sigma\|_{L^\infty(\D)}
    \Big( \sum_{T \in \mathcal{T}_h} \int_T | \nabla (R_h - I) u(z) |^2 \diff{z}
    \Big)^{\frac{1}{2}}\\
    &\quad \le C \| \sigma\|_{L^\infty(\D)} \big| (R_h - I) u \big|_{H^1(\D)}
    \le C \| \sigma\|_{L^\infty(\D)} \|u \|_{H^2(\D)} h 
  \end{align*}
  by a further application of \eqref{eq:rho}.

  For the estimate of the last term in \eqref{eqn:est1} we first consider
  $s \in (0,1)$. Applying H\"older's
  inequality with exponents $\rho = \frac{q}{2} \in (1,\infty)$ and $\rho' =
  \frac{q}{q-2} \in (1,\infty)$ yields
  \begin{align*}
    &\Big( \sum_{T \in \mathcal{T}_h} 
    |T|^{-2} \int_T \Big( \int_T \big( \sigma(x) - \sigma(z) \big) \diff{x}
    \Big)^2 | \nabla u(z) |^2 \diff{z} \Big)^{\frac{1}{2}}\\
    &\quad \le \Big( \sum_{T \in \mathcal{T}_h} 
    |T|^{-2} \Big( \int_T \Big( \int_T \big| \sigma(x) - \sigma(z) \big|
    \diff{x} \Big)^{2 \rho} \diff{z} \Big)^{\frac{1}{\rho}}
    \Big( \int_T | \nabla u(z) |^{2 \rho'} \diff{z} \Big)^{\frac{1}{\rho'}}
    \Big)^{\frac{1}{2}}\\
    &\quad \le \Big( \sum_{T \in \mathcal{T}_h} 
    |T|^{- \frac{1}{\rho}} \Big( \int_T \int_T \big| \sigma(x) - \sigma(z)
    \big|^{2 \rho} \diff{x} \diff{z} \Big)^{\frac{1}{\rho}}
    \Big( \int_T | \nabla u(z) |^{2 \rho'} \diff{z} \Big)^{\frac{1}{\rho'}}
    \Big)^{\frac{1}{2}}\\
    &\quad \le \Big( \sum_{T \in \mathcal{T}_h} 
    |T|^{- 1} \int_T \int_T \big| \sigma(x) - \sigma(z)
    \big|^{q} \diff{x} \diff{z} \Big)^{\frac{1}{q}}
    \Big( \sum_{T \in \mathcal{T}_h}
    \int_T | \nabla u(z) |^{2 \rho'} \diff{z} \Big)^{\frac{1}{2\rho'}}\\
    &\quad \le  \Big( \sum_{T \in \mathcal{T}_h} 
    |T|^{-1} h^{2 + qs} \int_T \int_T \frac{\big| \sigma(x) - \sigma(z)
    \big|^{q}}{|x-z|^{2 +qs}} \diff{x} \diff{z} \Big)^{\frac{1}{q}}
    \| u \|_{W^{1,2\rho'}(\D)} 
  \end{align*}
  since $|x-y| \le h$ for all $x,y \in T$.
  
  Next, recall that the Sobolev embedding theorem
  \cite[Theorem~4.12]{adams2003} yields
  \begin{align*}
    \| u \|_{W^{1,2\rho'}(\D)} \le  C \|u \|_{H^2(\D)}.
  \end{align*}
  In addition, we have
  $|T|^{-1} \le c^{-1} h^{-2}$ due to Assumption~\ref{as:triangulation}.
  Altogether, this shows 
  \begin{align*}
    &\Big( \sum_{T \in \mathcal{T}_h} 
    |T|^{-1} h^{2 + qs} \int_T \int_T \frac{\big| \sigma(x) - \sigma(z)
    \big|^{q}}{|x-z|^{2 +qs}} \diff{x} \diff{z} \Big)^{\frac{1}{q}}
    \| u \|_{W^{1,2\rho'}(\D)}\\
    &\quad \le C \|u \|_{H^2(\D)} | \sigma |_{W^{s,q}(\D)} h^s.
  \end{align*}
  This completes the proof of the case $s \in (0,1)$. The border case $s=1$
  follows by similar arguments and an additional application of the
  Poincar\'e--Wirtinger inequality. The details are left to the reader.
\end{proof}


\section{Variance reduction by importance sampling}
\label{sec:importanceSampling}

The goal of this section is to increase the accuracy of
the randomized quadrature formula $Q_{MC}$ introduced in \eqref{eq:MC} 
by applying a standard variance reduction technique for Monte Carlo methods
termed \emph{importance sampling}. 
An introduction to importance sampling and further variance reduction
techniques is found, for instance, in \cite[Chapter~6]{evans2000},
\cite[Chapter~3]{madras2002}, and \cite[Kapitel~5]{mueller2012}.

Let us briefly recall the main idea of importance sampling. Suppose one wants
to approximate the integral 
\begin{align*}
  \int_\D v(x) \diff{x},
\end{align*}
where $v \in L^2(\D)$ is given. Then, the standard Monte Carlo approach is to
rewrite the integral as an expectation
\begin{align*}
  \E[ v(Z) ] = |\D|^{-1} \int_{\D} v(x) \diff{x},
\end{align*}
where $Z \colon \Omega \to \D$ is a uniformly distributed random variable. In
particular, the probability density function of $Z$ is given by $p_Z(x) =
\frac{1}{|\D|}\one_{\D}(x)$. Then, the standard Monte Carlo estimator of the
integral is defined as
\begin{align*}
  \frac{|\D|}{M} \sum_{i = 1}^{M} v(Z_i),
\end{align*}
where $(Z_i)_{i = 1}^{M}$, $M \in \N$, is a family of independent and
identically distributed copies of $Z$. This estimator is unbiased and its
variance is equal to
\begin{align*}
  \Big\| \frac{|\D|}{M} \sum_{i = 1}^{M} v(Z_i) - \int_\D v(x) \diff{x}
  \Big\|_{L^2(\Omega)}^2 
  &= \frac{1}{M} \mathrm{var}\big(|\D| v(Z)\big).
\end{align*}
Therefore, the accuracy of the Monte Carlo estimator is determined by
the number of samples $M \in \N$ and the variance of the random variable $|\D|
v(Z)$. 

The main idea of importance sampling is then to increase the accuracy of the
standard Monte Carlo estimator by replacing the uniformly distributed random
variable $Z$ with a random variable $Y \colon \Omega \to \D$ whose distribution
is determined by a probability distribution function $p_Y$. If the density
$p_Y$ satisfies that $p_Y(x) = 0$ only if $v(x) = 0$,
then it follows from the transformation theorem that
\begin{align*}
  \int_\D v(x) \diff{x} = \int_\D \frac{v(x)}{p_Y(x)} p_Y(x) \diff{x}
  = \E \Big[ \frac{v(Y)}{p_Y(Y)} \Big].
\end{align*}
From this one derives the following \emph{importance sampling estimator}
given by
\begin{align*}
  \frac{1}{M} \sum_{i = 1}^{M} \frac{v(Y_i)}{p_Y(Y_i)},
\end{align*}
where $(Y_i)_{i = 1}^M$ denotes a family of independent and identically
distributed copies of $Y$.
The art of importance sampling is then to determine a suitable density $p_Y$
such that the variance is reduced and, at the same time, the generation of
random variates with density $p_Y$ is computational feasible and
affordable. It is known (cf. \cite[Theorem~6.5]{evans2000}) that the optimal
choice of the density $p_Y$ is
\begin{align*}
  p_Y^\ast(x) = \frac{|v(x)|}{\int_\D |v(y)|\diff{y}}, \quad \text{ for } x \in
  \D.
\end{align*}
Observe that $p_Y^\ast$ suggests to avoid sampling in regions of $|\D|$, where
$|v|$ is zero or very small. However, since the denominator is typically unknown 
it is, in general, not possible to use the density $p_Y^\ast$ in practice.

Nevertheless, one can often still make use of the underlying idea 
to improve the accuracy of the randomized
quadrature rule \eqref{eq:MC}. 
To demonstrate this, we solely focus on the Poisson equation
\begin{align}
  \label{eq:Poisson}
  \begin{cases}
    - \Delta u = f,& \quad \text{ in } \D,\\
    u = 0,& \quad \text{ on } \partial \D,
  \end{cases}
\end{align}
where $\D \subset \R^2$ is a convex, bounded and polygonal
domain and $f \in L^p(\D)$ for some $p\in [2,\infty]$. 

Observe that the Poisson equation is a particular case of the boundary
value problem \eqref{eq:BVP} with $\sigma \equiv 1$. In this case, 
the assembly of the stiffness matrix $A_h$ in \eqref{eq:matvec} does not
require the application of a (randomized) quadrature rule. 

Moreover, we recall that the entries of the 
load vector $f_h \in \R^{N_h}$ defined in
\eqref{eq:load} are given by 
\begin{align*}
  F(\varphi_j) = \int_{\D} f(x) \varphi_j(x) \diff{x}, \quad j \in
  \{1,\ldots,N_h\},
\end{align*}
where $(\varphi_j)_{j = 1}^{N_h}$ denotes the standard Lagrange basis of the
finite element space $S_h$. According to the results in the previous section,
these entries are then approximated by an application of the randomized 
quadrature formula \eqref{eq:MC} given by
\begin{align*}
  F_{MC}( \varphi_j) = Q_{MC}[ f \varphi_j ] 
  = \sum_{T \in \mathcal{T}_h} |T| f(Z_T) \varphi_j(Z_T)
\end{align*}
for every $j \in \{1,\ldots,N_h\}$. Observe that for each triangle $T \in
\mathcal{T}_h$ the term 
\begin{align}
  \label{eq:stdMC}
  |T| f(Z_T) \varphi_j(Z_T)
\end{align}
can be regarded as a standard
Monte Carlo estimator with only $M=1$ sample for the integral
\begin{align*}
  \int_T f(x) \varphi_j(x) \diff{x}.
\end{align*}
The idea of this section is to replace this term by a suitable importance
sampling estimator. 

Since we do not want to impose any additional assumption on $f$ it is,
as already mentioned above, not feasible to use the corresponding optimal
density function  $p_Y^\ast$ with $v = f \varphi_j$.
Instead, we recall that
the piecewise linear basis function $\varphi_j$ 
is equal to zero in two of the three vertices and equal to one in
the remaining vertex of every triangle
$T \in \mathcal{T}_h$ with $T \cap \mathrm{supp}(\varphi_j)
\neq \emptyset$. In particular, this implies $\varphi_j(x) \ge 0$ for every $x
\in T$. Further, it holds
\begin{align*}
  \int_T \varphi_j(x) \diff{x} = \frac{1}{3} |T|.
\end{align*}
Therefore, the mapping $p_{T,j} \colon \D \to [0,\infty)$ defined by
\begin{align}
  \label{eq:pTj}
  p_{T,j}(x) = 3 |T|^{-1} \varphi_j(x) \one_T(x), \quad x \in \D,
\end{align}
is a probability density function. By replacing $Z_T$ in \eqref{eq:stdMC} with
a random variable  $Y_{T,j} \sim p_{T,j}(x) \diff{x}$
we arrive at the corresponding importance sampling estimator (again with only 
$M = 1$ sample)
\begin{align*}
  \frac{f(Y_{T,j})\varphi_j(Y_{T,j})}{p_{T,j}(Y_{T,j})} =
  \frac{1}{3} |T| f(Y_{T,j}) 
\end{align*}
for the integral $\int_T f(x) \varphi_j(x) \diff{x}$.
Observe that the use of $Y_{T,j}$ significantly decreases the probability of
the integrand $f \varphi_j$ being evaluated at a point $x \in T$ close to
a vertex, where the basis function 
$\varphi_j$ is equal to zero. We discuss the simulation of the
random variable $Y_{T,j}$ in Section~\ref{sec:sampling}.

To sum up, this suggests to use the linear mapping
$F_{IS} \colon S_h \to L^2(\Omega)$ given by
\begin{align}
  \label{eq:FIS}
  F_{IS}(v_h) =  \frac{1}{3} \sum_{T \in \mathcal{T}_h} |T| 
  \sum_{\substack{j = 1\\ T \cap \mathrm{supp}(\varphi_j) \neq
  \emptyset}}^{N_h} v_j f(Y_{T,j})
\end{align}
for every $v_h = \sum_{j = 1}^{N_h} v_j \varphi_j \in S_h$. Hereby, 
$(Y_{T,j})_{T \in \mathcal{T}_h, j \in \{1,\ldots,N_h\}}$ is a
family of independent random variables with $Y_{T,j} \sim p_{T,j}(x) \diff{x}$.
In particular, the entries of the load vector $f_h$ are then approximated by
\begin{align*}   
  F_{IS}(\varphi_j) = \frac{1}{3} \sum_{\substack{T \in \mathcal{T}_h\\
  T \cap \mathrm{supp}(\varphi_j) \neq \emptyset}} |T| f(Y_{T,j})
\end{align*}
for every $j \in \{1,\ldots,N_h\}$. 
As the following lemma shows, the 
importance sampling estimator \eqref{eq:FIS} is unbiased and
convergent in the limit $h \to 0$.

\begin{lemma}
  \label{lem:QMCIS}
  Let $\mathcal{T}_h$ be an admissible triangulation with maximal edge length 
  $h \in (0,1]$. Then, for every $f \in L^1(\D)$ and $v_h \in S_h$
  it holds that
  \begin{align*}
    \E \big[ F_{IS}(v_h)  \big] = \int_\D f(x)v_h(x) \diff{x}.
  \end{align*}
  Further, if $f \in L^p(\D)$, $p \in [2,\infty]$, then it holds for every
  $v_h \in S_h$ that 
  \begin{align*}
    &\Big\| \int_\D f(x)v_h(x) \diff{x} - F_{IS}(v_h) \Big\|_{L^2(\Omega)}\\
    &\quad \le  \frac{1}{\sqrt[4]{12}}  h \| v_h
    \|_{L^\infty(\D)}^{\frac{2}{p}}
    \| f \|_{L^p(\D)} 
    \big( 2 h |v_h|_{H^1(\D)} + \|v_h\|_{L^2(\D)}
    \big)^{1 - \frac{2}{p}}.
  \end{align*}
  In addition, if $f \in W^{s,2}(\D)$ for some $s \in (0,1)$ then it
  holds for every $v_h \in S_h$ that 
  \begin{align*}
    \Big\| \int_\D f(x)v_h(x) \diff{x} - F_{IS}(v_h) \Big\|_{L^2(\Omega)}
    \le h^{1 + s} \|v_h\|_{L^\infty(\D)} |f|_{W^{s,2}(\D)}.
  \end{align*}
\end{lemma}

\begin{proof}
  Let $v_h = \sum_{j = 1}^{N_h} v_j \varphi_j \in S_h$ be arbitrary with
  coefficients $(v_j)_{j = 1}^{N_h} \subset \R$. Due to $Y_{T,j} \sim
  \frac{3}{|T|}\varphi_j(z) \one_T(z) \diff{z}$ for every 
  $T \in \mathcal{T}_h$ we have
  \begin{align*}
    \sum_{j = 1}^{N_h} v_j \E \Big[ \frac{ |T|}{3}  f(Y_{T,j}) \Big]
    = \sum_{j = 1}^{N_h} v_j  \frac{|T|}{3} \int_T f(z)\varphi_j(z)
    \frac{3}{|T|} \diff{z} 
    = \int_T f(z) v_h(z) \diff{z}.
  \end{align*}
  Then, the first assertion follows by summing over all triangles of the
  triangulation.

  Now, let $f \in L^2(\D)$ be arbitrary.
  In the same way as in the proof of Lemma~\ref{lem:QMC},
  the mean-square error is shown to be equal to 
  \begin{align*}
    &\Big\| \int_\D f(x) v_h(x) \diff{x} - F_{IS}(v_h)
    \Big\|_{L^2(\Omega)}^2\\ 
    &\quad =  \sum_{T \in \mathcal{T}_h}
    \sum_{\substack{j = 1\\T \cap \mathrm{supp}(\varphi_j)
    \neq \emptyset}}^{N_h} v_j^2\E \Big[ \Big|
    \int_T f(x)\varphi_j(x) \diff{x} - \frac{|T|}{3} f(Y_{T,j}) \Big|^2 \Big],
  \end{align*}
  due to the independence of the random variables $(Y_{T,j})_{T \in
  \mathcal{T}_h, j \in \{1,\ldots,N_h\}}$.
  
  Then, for every $j\in \{1,\ldots,N_h\}$ and $T \in \mathcal{T}_h$ with $ T
  \cap \mathrm{supp}(\varphi_j)\neq \emptyset$ 
  we make use of $Y_{T,j} \sim \frac{3}{|T|} \one_T(z)\varphi_j(z) \diff{z}$
  and the Cauchy--Schwarz inequality. This yields
  \begin{align*}
    &\E \Big[ \Big| \int_T f(x)\varphi_j(x) \diff{x} - \frac{|T|}{3} f(Y_{T,j})
    \Big|^2 \Big]\\ 
    &\quad = \frac{3}{|T|}\int_T \Big|\int_T f(x)\varphi_j(x)
    \diff{x} - \frac{|T|}{3}  f(z) \Big|^2 \varphi_j(z) \diff{z} \\
    &\quad = \frac{3}{|T|}\int_T \Big|\int_T
    (f(x)-f(z))\varphi_j(x) \diff{x} \Big|^2\varphi_j(z) \diff{z} \\
    &\quad \leq \int_T\int_T (f(x)-f(z))^2\varphi_j(x)\varphi_j(z)
    \diff{x}\diff{z}\\
    &\quad = \frac{2}{3} |T| \int_T |f(x)|^2 \varphi_j(x) \diff{x}
    - 2 \Big( \int_T f(x) \varphi_j(x) \diff{x} \Big)^2.
  \end{align*}
  We neglect the last term and insert this estimate into the mean-square error.
  An application of Weitzenb\"ock's inequality \eqref{eq:Weitzenb}
  then yields
  \begin{align}
    \label{eq:step1}
    \begin{split}
    &\Big\| \int_\D f(x)v_h(x) \diff{x} - F_{IS}(v_h)  \Big\|_{L^2(\Omega)}^2\\
    &\quad = \sum_{T \in \mathcal{T}_h}
    \sum_{\substack{j = 1\\T \cap \mathrm{supp}(\varphi_j)
    \neq \emptyset}}^{N_h} v_j^2
    \E \Big[ \Big|
    \int_T f(x)\varphi_j(x) \diff{x} - \frac{|T|}{3} f(Y_{T,j}) \Big|^2 \Big]\\
    &\quad \le  \frac{2}{3} \sum_{T \in \mathcal{T}_h}
    \sum_{\substack{j = 1\\T \cap \mathrm{supp}(\varphi_j)
    \neq \emptyset}}^{N_h} v_j^2 |T| \int_T |  f(x) |^2 \varphi_j(x) 
    \diff{x}\\
    &\quad \le  \frac{1}{2\sqrt{3}} h^2
    \sum_{T \in \mathcal{T}_h}
    \sum_{\substack{j = 1\\T \cap \mathrm{supp}(\varphi_j)
    \neq \emptyset}}^{N_h} v_j^2 \int_T |  f(x) |^2 \varphi_j(x) 
    \diff{x}.
    \end{split}
  \end{align}
  Now, we assume that $f \in L^p(\D)$ with $p \in [2,\infty]$.
  To every $v_h = \sum_{j = 1}^{N_h} v_j \varphi_j \in
  S_h$ we then associate a mapping ${v}^\circ_h \colon \D \to 
  \R$ defined by $v^\circ_h(x) = \sum_{T \in
  \mathcal{T}_h} v_T \one_T(x)$, where $v_T := v_h(z_T)$ and $z_T \in T$
  denotes the barycenter of $T \in \mathcal{T}_h$. Observe that $v^\circ_h$ is
  piecewise constant on each triangle. 

  For every $T \in
  \mathcal{T}_h$ and $j \in \{1,\ldots,N_h\}$ with $T \cap
  \mathrm{supp}(\varphi_j) \neq \emptyset$
  let $z_j \in \overline{T}$ be the uniquely determined node, 
  which satisfies $\varphi_j(z_j)=1$. Clearly, it holds $|z_j - z_T| \le h$.
  Since $v_h$ is affine linear we obtain  that
  \begin{align*}
    |v_j - v_T| = |v_h(z_j) - v_h(z_T)| \le | \nabla v_h(z_T)| h.
  \end{align*}
  Then, we continue the estimate of the mean-square error in \eqref{eq:step1} 
  by adding and subtracting the coefficients of $v_h^\circ$ as follows:
  For $\rho = \frac{p}{2} \in [1,\infty]$ let $\rho' = \frac{p}{p-2} \in
  [1,\infty]$ be the conjugated H\"older exponent determined by
  $\frac{1}{\rho} + \frac{1}{\rho'} = 1$, where we set $\frac{1}{\infty} = 0$.
  Then, we get 
  \begin{align*}
    v_j^2 &= |v_j|^{\frac{2}{\rho}} |v_j|^{\frac{2}{\rho'}}
    \le \max_i |v_i|^{ \frac{2}{\rho}}  \big( |v_j - v_T| + |v_T|
    \big)^{\frac{2}{\rho'}}\\
    &\le  \| v_h \|_{L^\infty(\D)}^{\frac{2}{\rho}} 
    \big( |\nabla v_h(z_T)| h + |v_T| \big)^{\frac{2}{\rho'}}.
  \end{align*}
  After inserting this into \eqref{eq:step1} we obtain
  \begin{align*}
    &\Big\| \int_\D f(x) v_h(x) \diff{x} - F_{IS}(v_h)  \Big\|_{L^2(\Omega)}^2\\
    &\quad \le  \frac{1}{2\sqrt{3}} h^2 \| v_h
    \|_{L^\infty(\D)}^{\frac{2}{\rho}}
    \sum_{T \in \mathcal{T}_h}
    \sum_{\substack{j = 1\\T \cap \mathrm{supp}(\varphi_j)
    \neq \emptyset}}^{N_h} \big( |\nabla v_h(z_T)| h + |v_T|
    \big)^{\frac{2}{\rho'}}
    \int_T |  f(x) |^2 \varphi_j(x) \diff{x}\\
    &\quad \le  \frac{1}{2\sqrt{3}} h^2 \| v_h
    \|_{L^\infty(\D)}^{\frac{2}{\rho}}
    \int_\D \big( |\nabla v_h(x)| h + |v_h^\circ(x)|  \big)^{\frac{2}{\rho'}}
    |  f(x) |^2 \diff{x},
  \end{align*}
  since $\nabla v_h$ and $v_h^\circ$ are constant on each $T$. In addition,
  we also made use of
  \begin{align}
    \label{eq:partion_1}
    0 \le \sum_{j = 1}^{N_h} \varphi_{j}(x) \le 1
  \end{align}
  for every $x \in \overline{\D}$.

  Therefore, applications of H\"older's inequality and Minkowski's inequality
  yield
  \begin{align}
    \label{eq:step3}
    \begin{split}
    &\Big\| \int_\D f(x) v_h(x) \diff{x} - F_{IS}(v_h)
    \Big\|_{L^2(\Omega)}^2\\
    &\quad \le \frac{1}{2\sqrt{3}} h^2 \| v_h
    \|_{L^\infty(\D)}^{\frac{2}{\rho}}
    \| f \|_{L^p(\Omega)}^{2} 
    \Big( \int_\D \big( |\nabla v_h(x)| h + |v_h^\circ(x)|  \big)^{2}
    \diff{x} \Big)^{\frac{1}{\rho'}}\\
    &\quad \le \frac{1}{2\sqrt{3}} h^2 \| v_h
    \|_{L^\infty(\D)}^{\frac{2}{\rho}}
    \| f \|_{L^p(\D)}^{2} 
    \big(  h |v_h|_{H^1(\D)} + \|v_h^\circ\|_{L^2(\D)}
    \big)^{\frac{2}{\rho'}}.
    \end{split}
  \end{align}
  Finally, we observe that
  \begin{align*}
    \| v_h - v_h^\circ \|_{L^2(\D)}^2
    &= \sum_{T \in \mathcal{T}_h} \int_T | v_h(x) - v_T|^2 \diff{x}\\
    &= \sum_{T \in \mathcal{T}_h} \int_T | \nabla v_h(x) \cdot (x - z_T) |^2
    \diff{x}
    \le h^2 | v_h |_{H^1(\D)}^2
  \end{align*}
  since $|x - z_T|\le h$ for every $x \in T$ and $\nabla v_h$ is piecewise
  constant on $T$. Consequently, 
  \begin{align*}
    \|v_h^\circ\|_{L^2(\D)} 
    \le \|v_h^\circ - v_h\|_{L^2(\D)}  + \|v_h\|_{L^2(\D)} 
    \le h | v_h |_{H^1(\D)} + \|v_h\|_{L^2(\D)}.
  \end{align*}
  Inserting this into \eqref{eq:step3} then completes the
  proof of the second assertion.

  To prove the third assertion let $f \in
  W^{s,2}(\D)$, $s\in (0,1)$. As above we have 
  \begin{align*}
    &\Big\| \int_\D f(x) v_h(x) \diff{x} - F_{IS}(v_h)
    \Big\|_{L^2(\Omega)}^2\\
    &\quad \le \sum_{T \in \mathcal{T}_h}
    \sum_{\substack{j = 1\\T \cap \mathrm{supp}(\varphi_j)
    \neq \emptyset}}^{N_h} v_j^2
    \int_T\int_T (f(x)-f(z))^2\varphi_j(x)\varphi_j(z)
    \diff{x}\diff{z}\\
    &\quad \le \max_{i} |v_i|^2 
    \sum_{T \in \mathcal{T}_h}
    \int_T \int_T (f(x)-f(z))^2 \diff{x}\diff{z},
  \end{align*}
  where we also used that $\varphi_j(z) \le 1$ for all $z \in T$ and
  \eqref{eq:partion_1}. Moreover, since $f \in W^{s,2}(\D)$ we get
  \begin{align*}
    \sum_{T \in \mathcal{T}_h} \int_T \int_T (f(x)-f(z))^2 \diff{x}\diff{z}
    &\le h^{2(1 + s)} \sum_{T \in \mathcal{T}_h}
    \int_T \int_T \frac{ |f(x)-f(z)|^2}{|x-z|^{2 + 2s}}
    \diff{x}\diff{z}\\
    &\le h^{2(1 + s)} |f|_{W^{s,2}(\D)}^2.
  \end{align*}
  Altogether, this completes the proof of the third assertion.
\end{proof}

The well-posedness of \eqref{eq:FIS} is a consequence of Lemma~\ref{lem:QMCIS}.
The following lemma contains some further estimates of $F_{IS}$ provided the
family of triangulations satisfies Assumption~\ref{as:triangulation}.

\begin{corollary} 
  \label{cor:wellposedness}
  Suppose that $f \in L^2(\D)$. 
  Let $(\mathcal{T}_{h})_{h \in (0,1]}$ be a family of triangulations
  satisfying Assumption~\ref{as:triangulation}. Then, 
  there exists $C \in (0,\infty)$ independent of $\mathcal{T}_h$ such that
  \begin{align*}
   |F_{IS}(v_h)| &\le C \ell_h^{\frac{1}{2}}
    \bar{F}_{IS,h} |v_h|_{H^1(\D)} < \infty \quad \text{$\P$-a.s.},\\
    \| F_{IS}(v_h) \|_{L^2(\Omega)} &\le C \|f \|_{L^2(\D)} |v_h|_{H^1(\D)} ,  
  \end{align*}
  for all $v_h \in S_h$, where $\ell_h = \max( 1 , \log(1/h))$ and 
  $\bar{F}_{IS,h} \colon \Omega \to \R$ is defined as
  \begin{align*}   
    \bar{F}_{IS,h} := \frac{1}{3} \sum_{T \in \mathcal{T}_h} |T|
    \sum_{\substack{j = 1 \\ T \cap \mathrm{supp}(\varphi_j) \neq
    \emptyset}}^{N_h} |f(Y_{T,j})|.
  \end{align*}   
\end{corollary}

\begin{proof}
  We only verify the almost sure bound for $F_{IS}(v_h)$. The estimate 
  of the $L^2(\Omega)$-norm then follows from Lemma~\ref{lem:QMCIS} and the
  same arguments as in the proof Lemma~\ref{lem:wellposedness}.

  By the definition of $F_{IS}$ and an application of \eqref{eq:maxnorm} we
  have that 
  \begin{align*}
    | F_{IS}(v_h) | 
    &\le \frac{1}{3} \sum_{T \in \mathcal{T}_h} |T| 
    \sum_{\substack{j = 1\\ T \cap \mathrm{supp}(\varphi_j) \neq
    \emptyset}}^{N_h} |v_j| |f(Y_{T,j})|\\
    &\le  \frac{1}{3} \| v_h \|_{L^\infty(\D)} 
    \sum_{T \in \mathcal{T}_h} |T| 
    \sum_{\substack{j = 1\\ T \cap \mathrm{supp}(\varphi_j) \neq
    \emptyset}}^{N_h}  |f(Y_{T,j})|\\
    &\le C \ell_h^{\frac{1}{2}} | v_h |_{H^1(\D)} \bar{F}_{IS,h}.
  \end{align*}
  It remains to show that $\bar{F}_{IS,h}$ is bounded $\mathbb{P}$-almost
  surely. But this follows immediately from
  \begin{align*}
    \E \big[ \bar{F}_{IS,h} \big]
    &= \frac{1}{3} \sum_{T \in \mathcal{T}_h} |T|
    \sum_{\substack{j = 1 \\ T \cap \mathrm{supp}(\varphi_j) \neq
    \emptyset}}^{N_h} \E \big[ |f(Y_{T,j})| \big]\\
    &=  \sum_{T \in \mathcal{T}_h} 
    \sum_{\substack{j = 1 \\ T \cap \mathrm{supp}(\varphi_j) \neq
    \emptyset}}^{N_h}  \int_T |f(y)| \varphi_j(y) \diff{y}\\
    &\le \int_\D |f(y)| \diff{y} < \infty,
  \end{align*}
  where we used that $\sum_{j = 1}^{N_h} \varphi_j(y) \le 1$ for every $y \in
  \D$. In turn, this implies $\bar{F}_{IS,h}< \infty$ $\mathbb{P}$-almost
  surely. 
\end{proof}

Next, we introduce the finite element problem based on the importance sampling
estimator. In terms of $F_{IS}$ the problem is stated as follows:
\begin{align}
  \label{eq:discProbis}
  \begin{cases}
    \text{Find } u_h^{IS} \colon \Omega \to S_h \text{ such that $\P$-almost
    surely} \\ 
    a(u_h^{IS},v_h) = F_{IS}(v_h) \text{ for all } v_h \in S_h.
  \end{cases}
\end{align}

In the same way as in Theorem~\ref{thm:existence} one shows that 
the discrete problem \eqref{eq:discProbis} has a uniquely determined solution
$u_h^{IS} \colon \Omega \to S_h$.

\begin{theorem}
  \label{thm:existenceis}
  For every admissible triangulation $\mathcal{T}_h$, $h \in (0,1]$, there
  exists a uniquely determined measurable mapping $u_h^{IS}
  \colon \Omega \to S_h$
  which solves the discrete problem \eqref{eq:discProb}. In addition, there
  exists $C \in (0,\infty)$ independent of $\mathcal{T}_h$ such that
  \begin{align*}
    \big|u_h^{IS}\big|_{H^1(\D)} 
    &\le C \ell_h^{\frac{1}{2}} \bar{F}_{IS,h} \quad \text{$\P$-a.s.,}
  \end{align*}
  where $\ell_h = \max(1,\log(1/h))$.
\end{theorem}

The following theorem contains an estimate of the total error of the
approximation $u_h^{IS}$ with respect to the $L^2(\Omega;H^1_0(\D))$-norm.

\begin{theorem}
  \label{thm:errorISH1}
  Let Assumptions~\ref{as:reg} and \ref{as:triangulation} be satisfied. If $f
  \in L^p(\D)$, $p \in (2,\infty]$, then there exists $C \in (0,\infty)$ such
  that for every $h \in (0,1]$ 
  \begin{align*}
    \big\| u_h^{IS} - u \big\|_{L^2(\Omega;H^1_0(\D))} 
    &\le C h \|u \|_{H^2(\D)} + C \ell_h^{\frac{1}{2} + \frac{1}{p}}
    h^{1 - \frac{2}{p}} \| f \|_{L^p(\D)},
  \end{align*}
  where $\ell_h = \max( 1 , \log(1/h))$.
\end{theorem}

\begin{proof}
  As in the proof of Theorem~\ref{thm:errorH1} we split the error into the
  two parts 
  \begin{align*}
    u_h^{IS} - u = u_h^{IS} - R_h u + R_h u - u =: \theta + \rho,
  \end{align*}
  where we recall the definition of the Ritz projector $R_h \colon H^1_0(\D) \to
  S_h$ from Section~\ref{sec:notation}. Since 
  the associated bilinear form $a$ for \eqref{eq:Poisson}
  coincides with the inner product in $H^1_0(\D)$ it follows that
  \begin{align*}
    \big|u_h^{IS} - u \big|^2_{H^1(\D)} 
    &= a( u_h^{IS} - u, u_h^{IS} - u)
    = a( \theta, \theta) + a(\rho, \rho)\\
    &=  |\theta|^2_{H^1(\D)} + |\rho|^2_{H^1(\D)}.
  \end{align*}
  Then, due to \eqref{eq:RitzH1} it holds
  \begin{align*}
    | \rho |_{H^1(\D)} = | R_h u - u|_{H^1(\D)} \le C h \|u\|_{H^2(\D)}.
  \end{align*}
  Further, from the variational formulation of \eqref{eq:Poisson} and
  \eqref{eq:discProbis} we get $\P$-almost surely for every $v_h \in S_h$ that
  \begin{align*}
    a(\theta, v_h) &= a(u_h^{IS}, v_h) - a(R_h u, v_h)\\
    &= F_{IS}(v_h) - F(v_h),
  \end{align*}
  since $a(R_h u, v_h)= a(u,v_h) = F(v_h)$ for every $v_h \in S_h$.
  In particular, for the choice $v_h = \theta(\omega) = u_h^{IS}(\omega) 
  - R_h u \in S_h$ we obtain $\P$-almost surely that
  \begin{align*}
    |\theta|^2_{H^1(\D)} = a(\theta, \theta)
    = F_{IS}(\theta) - F(\theta).
  \end{align*}
  From Corollary~\ref{cor:wellposedness} and Theorem~\ref{thm:existenceis} it
  follows directly that all terms on the right-hand side are integrable with
  respect to $\P$. Hence, after taking expectations it remains to 
  prove an estimate for the term
  \begin{align*}
    E_{IS} &= \big| \E\big[ F_{IS}(\theta) - F(\theta) \big] \big|.
  \end{align*}
  This is accomplished by the same arguments as in the proof of
  Lemma~\ref{lem:FMC}. More precisely, we represent $\theta$ in terms of
  an orthonormal basis $(\psi_j)_{j = 1}^{N_h} \subset S_h$ by
  \begin{align*}
    \theta = \sum_{j = 1}^{N_h} \theta_j \psi_j,
  \end{align*}
  where $\theta_j= (\theta, \psi_j)_{L^2(\D)}$, $j = 1, \ldots, N_h$,
  are real-valued and square-integrable random variables.
  Hereby, we assume again that $(\psi_j)_{j = 1}^{N_h}$ is
  a solution to the discrete eigenvalue problem \eqref{eq:eigval}.
  Then, by the linearity of $F$ and $F_{IS}$ and the Cauchy--Schwarz inequality
  we obtain the estimate
  \begin{align*}
    E_{IS} &= \Big| \E \Big[ \sum_{j = 1}^{N_h} \theta_j \big( F_{IS}(\psi_j)
    - F(\psi_j) \big) \Big] \Big|\\
    &\le \Big(\sum_{j =1}^{N_h} \lambda_{h,j} \E\big[ |\theta_j|^2 \big]
    \Big)^{\frac{1}{2}}
    \Big( \sum_{j =1}^{N_h} \lambda_{h,j}^{-1}
    \E \big[ \big| F_{IS}(\psi_j) - F(\psi_j) \big|^2 \big]
    \Big)^{\frac{1}{2}},
  \end{align*}
  where $(\lambda_{h,j})_{j = 1}^{N_h} \subset (0,\infty)$ denote the
  discrete eigenvalues in \eqref{eq:eigval}.
  Then, as in \eqref{eq:v_hH1} one computes
  \begin{align*}
    \Big(\sum_{j =1}^{N_h} \lambda_{h,j} \E\big[ |\theta_j|^2 \big]
    \Big)^{\frac{1}{2}} = \| \theta \|_{L^2(\Omega;H^1_0(\D))}.
  \end{align*}
  Moreover, since $f \in L^p(\D)$ and $\|\psi_j\|_{L^2(\D)}=1$ 
  it follows from Lemma~\ref{lem:QMCIS} that
  \begin{align*}
    \E \big[ \big| F_{IS}(\psi_j) - F(\psi_j) \big|^2 \big]
    \le \frac{1}{\sqrt{12}} h^2 \|f\|_{L^p(\D)}^{2} 
    \| \psi_j \|^{\frac{4}{p}}_{L^\infty(\D)} 
    \big( 2 h |\psi_j|_{H^1(\D)} + 1 \big)^{2 - \frac{4}{p}}.
  \end{align*}
  Next, we recall from \eqref{eq:maxnorm} and \eqref{eq:inverse}
  that $\| \psi_j \|_{L^\infty(\D)} \le C \ell_h^{\frac{1}{2}}
  |\psi_j|_{H^1(\D)} \le C \ell_h^{\frac{1}{2}} h^{-1}$
  for every $j \in \{1,\ldots,N_h\}$, 
  since $\| \psi_j\|_{L^2(\D)} = 1$. Therefore,
  \begin{align*}
    \E \big[ \big| F_{IS}(\psi_j) - F(\psi_j) \big|^2 \big]
    \le C \ell_h^{\frac{2}{p}} h^{2 - \frac{4}{p}}
    \|f\|_{L^p(\D)}^{2}
  \end{align*}
  for some constant $C \in (0,\infty)$ independent of $h \in (0,1]$
  and $j \in \{1,\ldots,N_h\}$.
  
  Altogether, we have shown that
  \begin{align*}
    E_{IS} \le C \ell_h^{\frac{1}{p}} h^{1 - \frac{2}{p}}
    \|f\|_{L^p(\D)} \| \theta \|_{L^2(\Omega;H^1_0(\D))} 
    \Big( \sum_{j = 1}^{N_h} \lambda_{h,j}^{-1} \Big)^{\frac{1}{2}}.
  \end{align*}
  Together with \eqref{eq:sumlam-1} this completes the proof.
\end{proof}

Finally, we also show an error estimate with respect to the 
norm in $L^2(\Omega;L^2(\D))$.

\begin{theorem}
  \label{thm:errorISL2}
  Let Assumptions~\ref{as:reg} and \ref{as:triangulation} be satisfied. If $f
  \in W^{s,2}(\D)$, $s \in [0,1)$, then there exists $C \in (0,\infty)$ such
  that for every $h \in (0,1]$ 
  \begin{align*}
    \big\| u_h^{IS} - u \big\|_{L^2(\Omega;L^2(\D))} 
    &\le C h^2 \|u \|_{H^2(\D)} + C \ell_h
    h^{1 + s} | f |_{W^{s,2}(\D)},
  \end{align*}
  where $\ell_h = \max( 1 , \log(1/h))$.
\end{theorem}

\begin{proof}
  As in the proof of Theorem~\ref{thm:errorISH1} we again split the error into
  the two parts 
  \begin{align*}
    u_h^{IS} - u = u_h^{IS} - R_h u + R_h u - u =: \theta + \rho.
  \end{align*}
  Then, it follows from \eqref{eq:RitzL2} that
  \begin{align*}
    \| \rho \|_{L^2(\D)}
    = \| (R_h - I) u \|_{L^2(\D)}
    \le C h^2 \| u \|_{H^2(\D)}
  \end{align*}
  for every $h \in (0,1]$.

  In order to give an estimate of the $L^2(\Omega;L^2(\D))$-norm of $\theta$ 
  we apply Nitsche's duality trick. More precisely, we consider the auxiliary
  problem of finding a random mapping $w_h \colon  \Omega \to S_h$ satisfying
  $\P$-almost surely 
  \begin{align}
    \label{eq:dualprob}
    a(v_h,w_h) = (\theta, v_h)_{L^2(\D)}, \quad \text{ for all } v_h \in S_h.
  \end{align}
  Observe that \eqref{eq:dualprob} is a linear
  variational problem with a random right-hand side. The existence of a
  uniquely determined solution $w_h \colon \Omega \to S_h$ can be shown in the
  same way as in the proof of Theorem~\ref{thm:existence}.

  Testing \eqref{eq:dualprob} with $v_h = \theta(\omega) \in S_h$ then gives
  for $\P$-almost every $\omega \in \Omega$ that
  \begin{align*}
    \| \theta(\omega) \|_{L^2(\D)}^2 &= a(\theta(\omega),w_h(\omega))
    = a(u_h^{IS}(\omega), w_h(\omega)) - a(R_h u, w_h(\omega))\\
    &= F_{IS}(w_h(\omega)) - F(w_h(\omega)),
  \end{align*}
  where we also applied \eqref{eq:discProbis}, \eqref{eq:varprob2d}, and
  \eqref{eq:Ritz}. Therefore, we have
  \begin{align*}
    \| \theta \|_{L^2(\Omega;L^2(\D))}^2 
    = \big| \E \big[ F_{IS}(w_h) - F(w_h) \big] \big|.    
  \end{align*}
  Then, as in the proof of Lemma~\ref{lem:FMC} we represent $w_h$ in terms of
  the orthonormal basis $(\psi_j)_{j = 1}^{N_h}$ consisting of discrete 
  eigenfunctions to the eigenvalue problem \eqref{eq:eigval}.
  After inserting this into the $L^2(\Omega;L^2(\D))$-norm of $\theta$,
  an application of the Cauchy--Schwarz inequality yields
  \begin{align*}
    \| \theta \|_{L^2(\Omega;L^2(\D))}^2 
    &= \Big| \E \Big[ \sum_{j = 1}^{N_h} w_j \big( F_{IS}( \psi_j) - F(\psi_j)
    \big) \Big] \Big|\\ 
    &\le \Big(\sum_{j =1}^{N_h} \lambda_{h,j}^2 \E\big[ |w_j|^2 \big]
    \Big)^{\frac{1}{2}}
    \Big( \sum_{j =1}^{N_h} \lambda_{h,j}^{-2}
    \E \big[ \big| F_{IS}(\psi_j) - F(\psi_j) \big|^2 \big]
    \Big)^{\frac{1}{2}},
  \end{align*}
  where $w_j = (\psi_j,w_h)_{L^2(\D)}$, $j \in \{1,\ldots,N_h\}$, and
  $(\lambda_{h,j})_{j = 1}^{N_h} \subset (0,\infty)$ are the
  discrete eigenvalues in \eqref{eq:eigval}.

  Then, it follows from \eqref{eq:eigval}, \eqref{eq:dualprob} 
  and Parseval's identity that
  \begin{align*}
    \Big(\sum_{j =1}^{N_h} \lambda_{h,j}^2 \E\big[ |w_j|^2 \big]
    \Big)^{\frac{1}{2}}
    &= \Big(\sum_{j =1}^{N_h} \E \big[ \big| \lambda_{h,j} (\psi_j, w_h)_{L^2(\D)}
    \big|^2 \big] \Big)^{\frac{1}{2}}\\
    &= \Big(\sum_{j =1}^{N_h} \E \big[ \big| a(\psi_j, w_h) \big|^2 \big]
    \Big)^{\frac{1}{2}}\\
    &= \Big(\sum_{j =1}^{N_h} \E \big[ \big| (\theta,\psi_j)_{L^2(\D)} \big|^2
    \big] \Big)^{\frac{1}{2}}
    = \| \theta \|_{L^2(\Omega;L^2(\D))}.
  \end{align*}
  Hence, this term can be cancelled from both sides of the inequality.

  Furthermore, an application of Lemma~\ref{lem:QMCIS} shows that
  \begin{align*}
    \E \big[ \big| F_{IS}(\psi_j) - F(\psi_j) \big|^2 \big]
    &\le h^{2(1+s)} \| \psi_j \|_{L^\infty(\D)}^2 | f |_{W^{s,2}(\D)}^2. 
  \end{align*}
  After recalling from \eqref{eq:maxnorm} and \eqref{eq:eigval} that
  \begin{align*}
    \| \psi_j \|_{L^\infty(\D)}^2 \le C \ell_h | \psi_j|_{H^1(\D)}^2
    = C \ell_h a(\psi_j,\psi_j) =  C \ell_h \lambda_{h,j}
  \end{align*}
  for every $j\in\{1,\ldots,N_h\}$, we finally arrive at
  \begin{align*}
    \Big( \sum_{j =1}^{N_h} \lambda_{h,j}^{-2}
    \E \big[ \big| F_{IS}(\psi_j) - F(\psi_j) \big|^2 \big]
    \Big)^{\frac{1}{2}}
    &\le C \ell_h^{\frac{1}{2}} h^{1+s} | f |_{W^{s,2}(\D)}
    \Big( \sum_{j = 1}^{N_h} \lambda_{h,j}^{-1} \Big)^{\frac{1}{2}}\\
    &\le C \ell_h h^{1+s} | f |_{W^{s,2}(\D)},
  \end{align*}
  where we also inserted \eqref{eq:sumlam-1} in the last step.
  Altogether, this completes the proof for $s\in (0,1)$. The boarder case $s
  =0$ is proven analogously.
\end{proof}

%
%

\section{Implementation of the randomized quadrature formulas}
\label{sec:sampling}
This section is devoted to a brief instruction on how to implement the
randomized quadrature formulas \eqref{eq:MC} and \eqref{eq:FIS}.

To be more precise, we apply the \emph{general rejection algorithm} to sample
the random variables $Y_{T,j} \sim p_{T,j}(x) \diff{x}$ introduced in
\eqref{eq:pTj} for each element $T \in \mathcal{T}_h$ and $j \in
\{1,\ldots,N_h\}$.  
We briefly review the rejection algorithm in Section~\ref{sec:rejection}.
To simplify its implementation it is convenient to use a change of
coordinates such that the sampling can be done on a fixed reference
triangle. This will be discussed in detail in
Section~\ref{sec:transformation}. In Section~\ref{sec:samplingref} we
then show how the required samples are generated on the reference
triangle using the rejection algorithm. Moreover, Section~\ref{sec:uniform}
briefly considers the uniform sampling of $Z_T \sim \mathcal{U}(T)$ 
on an arbitrary triangle $T \in \mathcal{T}_h$. Finally, in
Section~\ref{sec:FEM} we sketch how the randomized quadrature formula
\eqref{eq:MC} can be embedded into the finite element method.

\subsection{General rejection algorithm}
\label{sec:rejection}

In this subsection we briefly recall the general rejection algorithm
for the simulation of a non-uniformly distributed random variable whose
distribution is given by a probability density function. For more details
on this method we refer to \cite[Chapter~2.3.2]{madras2002}. 

For $d \in \N$ let $p \colon \mathbb{R}^d \to \mathbb{R}$ be a given
probability density function. The goal is to generate samples of a random variable
$X \colon \Omega \to \R^d$ whose distribution is given by $p(x) \diff{x}$.
To this end, we assume that we already know how to generate samples of a random
variable $Z \colon \Omega \to \R^d$ which is distributed according to a further
probability density function $g \colon \R^d \to \R$. Suppose that
there exists $c \in (0,\infty)$ such that 
\begin{align}
  \label{eqn:rejection}
  p(x) \leq c g(x),\quad \text{ for all } x\in \mathbb{R}^d.
\end{align}
Then, the \emph{general rejection algorithm} is given by:
\begin{enumerate}
  \item[1.] Generate a sample $Z \sim g(x)\diff{x}$.
  \item[2.] Generate a sample $Y\sim \mathcal{U}(0,c)$ independently from $Z$.
  \item[3.] Return the value of $Z$ if $Y\cdot g(Z)\leq p(Z)$, otherwise go
    back to Step 1. 
\end{enumerate}
It can be shown that the output of the algorithm is distributed
according to the density $p$. Moreover, the expected number of 
samples of $(Z,Y)$ needed until a value of $Z$ is accepted is equal to $c$.
It is therefore desirable to choose $c$ in \eqref{eqn:rejection}
as small as possible. 
For a proof we refer to \cite[Theorem~2.15]{madras2002}.

\subsection{Transformation to a reference triangle}
\label{sec:transformation}

In this subsection we describe how to generate a sample of a random variable
whose distribution depends on a specific triangle $T$ of a given triangulation
$\mathcal{T}_h$ by making use of a transformation to a reference triangle. 
The same approach is widely used in practice for the assembly of the stiffness
matrix \eqref{eq:matvec} and can therefore easily be added to existing code.

We purely focus on generating samples of the random variables $Y_{T,j}$, $T \in
\mathcal{T}_h$, $j \in \{1,\ldots,N_h\}$, introduced in
Section~\ref{sec:importanceSampling}. Recall that the 
probability density function associated to $Y_{T,j}$ is given by
\begin{align*}
  p_{T,j}(x) = 3 |T|^{-1} \varphi_j(x) \one_{T}(x), \quad x \in \D \subset \R^2.
\end{align*}
Let us fix a triangle $T \in \mathcal{T}_h$ with vertices $(x_1,y_1)$,
$(x_2,y_2)$ and $(x_3,y_3)$, such that $T \cap \mathrm{supp}(\varphi_j)\neq
\emptyset $. Without loss of generality we assume that $\varphi_j(x_1,y_1) =
1$.

We want to use the general rejection algorithm in order to generate samples 
of $Y_{T,j}$. However, the probability density function $p_{T,j}$
depends on the specific triangle and the basis function $\varphi_j$.
Since it is inconvenient to set up the rejection method for each element and
basis function separately, we will now describe in detail, how to simplify
this problem by using a so called isoparametric transformation denoted by
$\Gamma \colon T \to S_2$. Hereby, $S_2 \subset \R^2$ denotes the standard
$2$-simplex.   

\begin{figure}
  \includegraphics[height=5cm, trim={0.5cm 19cm 0.5cm
  2cm},clip]{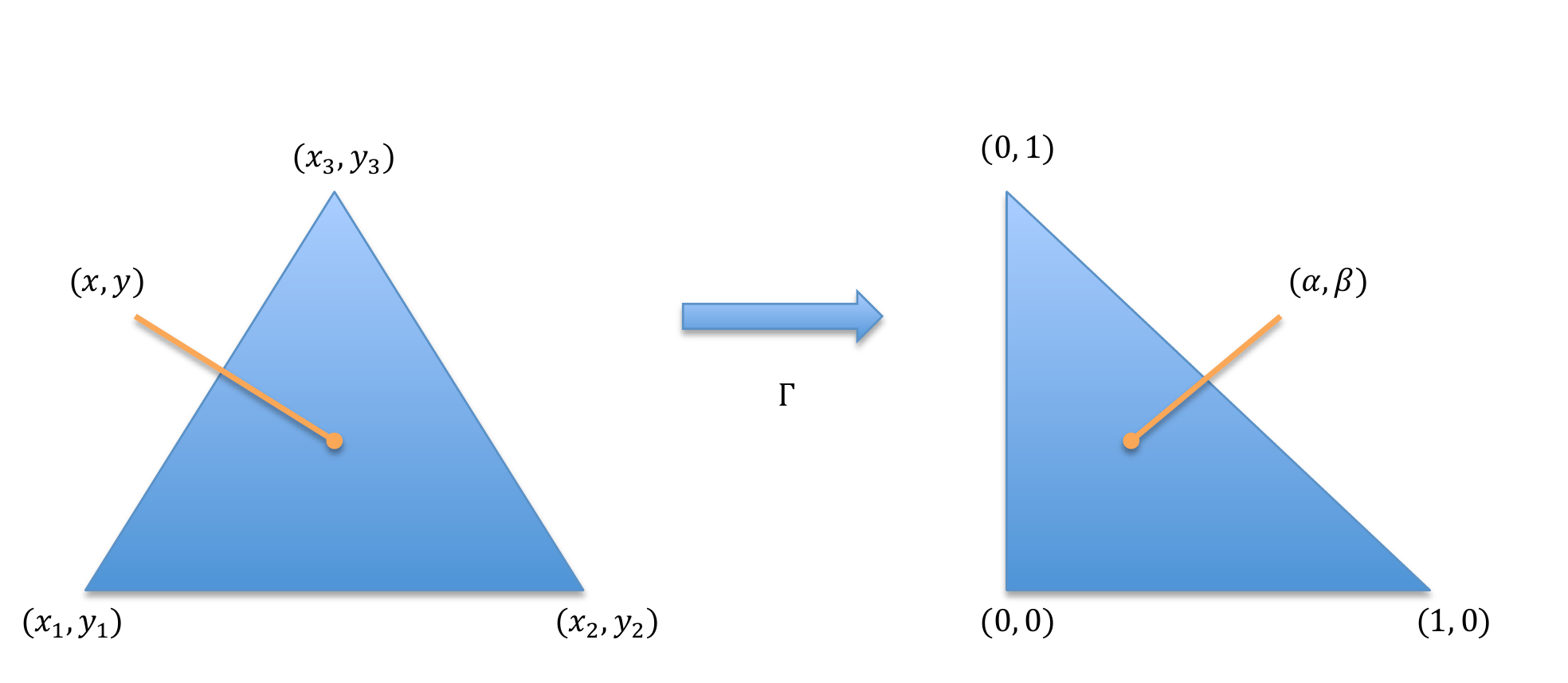} 
  \caption{Triangle transformation to the standard $2$-simplex, where $(x,y)$
  and $(\alpha,\beta) = \Gamma(x,y)$ represent interior points of the
  respective triangles.}
  \label{fig:1}
\end{figure}

As illustrated in Figure~\ref{fig:1} we denote the coordinates of a point in 
the given triangle $T$ by $(x,y)$, while
the ones in the standard $2$-simplex $S_2$ are written as $(\alpha,\beta)$.
Then, the coordinate transformation $\Gamma \colon T \to S_2$ is given by
\begin{align*}
  \begin{bmatrix}
    \alpha\\
    \beta
  \end{bmatrix}
  =\Gamma(x,y):=
  \begin{bmatrix}
    x_2-x_1 & x_3-x_1\\
    y_2-y_1 & y_3-y_1\\
  \end{bmatrix}^{-1}
  \begin{bmatrix}
    x-x_1\\
    y-y_1
  \end{bmatrix},
\end{align*}
while the inverse $\Gamma^{-1} \colon S_2 \to T$ is explicitly determined by
\begin{align}
  \label{eqn:inversetransformation}
  \begin{bmatrix}
    x\\
    y
  \end{bmatrix}
  = \Gamma^{-1}(\alpha,\beta)
  :=
  \begin{bmatrix}
    x_2-x_1 & x_3-x_1\\
    y_2-y_1 & y_3-y_1\\
  \end{bmatrix}
  \begin{bmatrix}
    \alpha\\
    \beta
  \end{bmatrix}
  +
  \begin{bmatrix}
    x_1\\
    y_1
  \end{bmatrix}.
\end{align}
Observe that $\Gamma^{-1}(0,0) = (x_1,y_1)$.

Next, we consider the mapping $\hat{\varphi} \colon S_2 \to \R$ defined by
\begin{align}\label{eq:varphi_transform1}
  \hat{\varphi}(\alpha,\beta) = 1 - \alpha - \beta, \quad \text{for all }
  (\alpha, \beta) \in S_2.
\end{align}
Since $\hat{\varphi}$ is affine linear one easily verifies that
\begin{align*}
  \hat{\varphi}(\alpha,\beta) = \varphi_j( \Gamma^{-1}(\alpha,\beta)), 
  \quad \text{for all } (\alpha, \beta) \in S_2.
\end{align*}
Moreover, it holds
\begin{align*}
  \int_{S_2} \hat{\varphi}(\alpha,\beta) \diff{(\alpha,\beta)} =
  \frac{1}{3}|S_2| = \frac{1}{6}.
\end{align*}
Therefore, the mapping $\hat{p} \colon \R^2 \to \R$ given by
\begin{align}
  \label{eq:phat}
  \hat{p}(\alpha,\beta) = 6 \hat{\varphi}(\alpha,\beta)
  \one_{S_2}(\alpha,\beta), \quad \text{for } (\alpha, \beta) \in \R^2,
\end{align}
is a probability density function. 
Suppose that $\hat{Y} \colon \Omega \to \R^2$ is a random variable with 
distribution $\hat{p}(\alpha,\beta)\diff{(\alpha,\beta)}$.
Then, it follows that 
\begin{align*}
  Y_{T,j} \sim \Gamma^{-1}( \hat{Y}),
\end{align*}
i.e. both random variables are identically distributed with
the probability density function $p_{T,j}$. In fact, for every $B \in \B(\R^2)$
it holds
\begin{align*}
  \P( \{ \Gamma^{-1}( \hat{Y}) \in B \} )
  = \P( \{ \hat{Y} \in \Gamma(B) \} )
  = \int_{\Gamma(B)} \hat{p}(\alpha,\beta) \diff{(\alpha,\beta)}.
\end{align*}
After inserting $\hat{p}$ and since $\Gamma(B) \cap S_2 = \Gamma(B \cap T)$
we arrive at
\begin{align*}
  \P( \{ \Gamma^{-1}( \hat{Y}) \in B \} )
  &= 6 \int_{\Gamma(B)} 
  \hat{\varphi}(\alpha,\beta) \one_{S_2}(\alpha,\beta)
  \diff{(\alpha,\beta)}
  = 6 \int_{\Gamma(B \cap T)} \hat{\varphi}(\alpha,\beta)
  \diff{(\alpha,\beta)}\\
  &= 6 \int_{B \cap T} \hat{\varphi}( \Gamma(x,y) ) |\det(D\Gamma)(x,y)|
  \diff{(x,y)}\\
  &= 6 \int_{B} \varphi_j( x,y) \one_T(x,y) |\det(D\Gamma)(x,y)|
  \diff{(x,y)}
\end{align*}
by a change of coordinates. 
Since $\Gamma$ is affine linear, the Jacobian $D \Gamma \in \R^{2,2}$ is
constant and the determinant is easily computed as 
\begin{align*}
  |\det(D\Gamma)| = |\det(D\Gamma^{-1})|^{-1} =  \frac{1}{2|T|}.
\end{align*}
Therefore,
\begin{align*}
  \P( \{ \Gamma^{-1}( \hat{Y}) \in B \} )
  = \frac{3}{|T|} \int_{B} \varphi_j( x,y) \one_T(x,y) \diff{(x,y)}
  = \int_B p_{T,j}(x,y) \diff{(x,y)}.
\end{align*}
Consequently, in order to generate a sample of the random variable
$Y_{T,j} \sim p_{T,j}$ it is sufficient to generate a sample of $\hat{Y} \sim
\hat{p}$ and to apply the transformation $\Gamma^{-1}$.

In addition, for the cases of $\varphi_j(x_2,y_2)=1$ or $\varphi_j(x_3,y_3)=1$,
if using the same triangle transform as illustrated in Figure \ref{fig:1}, the
only step that differs from the above description is in
\eqref{eq:varphi_transform1}. It needs to be changed accordingly to
\begin{align*}
  \hat{\varphi}(\alpha,\beta) = \alpha, \quad \text{for all }
  (\alpha, \beta) \in S_2,
\end{align*}
if $\varphi_j(x_2,y_2)=1$, or
\begin{align*}
  \hat{\varphi}(\alpha,\beta) = \beta, \quad \text{for all }
  (\alpha, \beta) \in S_2,
\end{align*}
in the case of $\varphi_j(x_3,y_3)=1$.

\subsection{Generating samples of $\hat{Y}$ on the reference triangle} 
\label{sec:samplingref}

It remains to discuss how to generate samples of the random variable
$\hat{Y} \sim \hat{p}(\alpha,\beta)\diff{(\alpha,\beta)}$ introduced in
\eqref{eq:phat}. To this end, we apply the general rejection algorithm from
Section~\ref{sec:rejection} with 
\begin{align*}
  g(\alpha,\beta)=2\one_{S_2}(\alpha,\beta), \quad \text{for } (\alpha,\beta)
  \in \R^2,  
\end{align*}
as the probability density function of the random variable $Z$, i.e. $Z \sim
\mathcal{U}(S_2)$. We recall that $S_2 \subset \R^2$ denotes the standard
$2$-simplex. We also define
\begin{align*}
  c:= \sup\Big\{\frac{\hat{p}(\alpha,\beta)}{g(\alpha,\beta)}\, 
  | \, (\alpha,\beta) \in S_2\Big\}
  = \sup\Big\{  \frac{1}{2} \hat{p}(\alpha,\beta) \, 
  | \, (\alpha,\beta) \in S_2\Big\} = 3.
\end{align*}
Then, \eqref{eqn:rejection} is satisfied. Therefore, the general rejection
algorithm is applicable and generates samples of $\hat{Y} \sim
\hat{p}(\alpha,\beta)\diff{(\alpha,\beta)}$ as follows: 
\begin{enumerate}
  \item[1.] Generate $Z = (Z_1,Z_2) \sim \mathcal{U}(S_2)$ as follows:
    \begin{enumerate}
      \item Generate $U_1, U_2 \sim \mathcal{U}(0,1)$ independently.
      \item If $U_1 + U_2 \le 1$ then set $Z := (U_1,U_2)$,
        else set $Z := (1-U_1, 1 -U_2)$.
    \end{enumerate}
  \item[2.] Generate $Y \sim \mathcal{U}(0,c)$ independently of $Z$.
  \item[3.] Output $Z = (Z_1,Z_2)$ if $Y g(Z_1,Z_2) \leq \hat{p}(Z_1,Z_2)$,
    else go back to Step 1. 
\end{enumerate}

\begin{remark}
  As an alternative to the rejection method 
  one could generate samples of $\hat{Y}= (\hat{Y}_1, \hat{Y}_2)$ by first
  applying the inversion method, cf.~\cite[Chapter~2]{madras2002}, for the
  simulation of the marginal distribution of the first variable $\hat{Y}_1$.
  Thereafter, a further application of the inversion method can be used
  to generate a sample of $\hat{Y}_2$ conditional on the already generated
  sample of $\hat{Y}_1$. Depending on the actual implementation, this could be
  more efficient. However, this approach is much harder to
  generalize to other probability density functions or to higher dimensional
  domains.
\end{remark}

\subsection{Generating uniformly distributed samples on arbitrary elements}
\label{sec:uniform}

In this subsection, we briefly discuss the generation of uniformly distributed 
random variables $Z_T \sim \mathcal{U}(T)$ for an arbitrary triangle $T \in
\mathcal{T}_h$. These random variables are required for the randomized
quadrature formula \eqref{eq:MC}.
This is easily accomplished by making use of the results from
the previous two subsections. Indeed, we just have to generate 
a sample of a uniformly distributed random variable $Z \sim \mathcal{U}(S_2)$,
where $S_2$ again denotes the $2$-simplex. Then, we apply the corresponding 
inverse transformation $\Gamma^{-1}$ from \eqref{eqn:inversetransformation}
associated to the given triangle $T \in \mathcal{T}_h$.
As a result, we obtain $Z_T = \Gamma^{-1}(Z) \sim \mathcal{U}(T)$ for $T \in
\mathcal{T}_h$.

The sampling procedure is summarized in the following two steps.
\begin{enumerate}
  \item[1.] Generate $Z = (Z_1,Z_2) \sim \mathcal{U}(S_2)$ as follows:
    \begin{enumerate}
      \item Generate $U_1, U_2 \sim \mathcal{U}(0,1)$ independently.
      \item If $U_1 + U_2 \le 1$ then set $Z := (U_1,U_2)$,
        else set $Z := (1-U_1, 1 -U_2)$.
    \end{enumerate}
  \item[2.] Output: $Z_T = \Gamma^{-1}(Z_1,Z_2)$, where $\Gamma^{-1}$ in
    \eqref{eqn:inversetransformation} uses the coordinates of the vertices of
    $T$.
\end{enumerate}

\subsection{Implementation of the FEM with randomized quadrature formulas}
\label{sec:FEM}
In this part, we illustrate the implementation of the finite element method
with the randomized quadrature formula \eqref{eq:MC} for the 
elliptic equation \eqref{eq:BVP}. The implementation of \eqref{eq:discProbis}, 
which is based on the importance sampling estimator, can be done in a similar
way.

Algorithm~\ref{alg1} lists one possibility to compute a realization of the
numerical approximation of the solution to \eqref{eq:BVP}
based on the Monte Carlo estimator \eqref{eq:MC}.  

\begin{algorithm}
  \caption{FEM with MC estimator \eqref{eq:MC} for the elliptic equation
  \eqref{eq:BVP}}
  \label{alg1}
  \begin{algorithmic}[1] 
    \State \textbf{Input}: 
    $\mathcal{T}_h$ triangulation of domain $\D$,
    functions $f$ and $\sigma$;
    \State Get the set of interior nodes $(z_j)_{j=1}^{N_h}$ of $\mathcal{T}_h$
    with associated Lagrange basis functions $(\varphi_j)_{j=1}^{N_h}$; 
    \State Generate $Z_T^{1},Z^{2}_T \sim \mathcal{U}(T)$ independently
    for every $T \in \mathcal{T}_h$ (see Section~\ref{sec:uniform});
    \State Compute the function values $(\sigma(Z_T^{1}))_{T \in
    \mathcal{T}_h}$ and $(f(Z_T^{2}))_{T \in \mathcal{T}_h}$;
    \State Assemble the stiffness matrix $A_{MC}$ with entries
    $\big(a_{MC}(\varphi_{k_1},\varphi_{k_2})\big)_{k_1,k_2=1}^{N_h}$ 
    based on the values $(Z_T^1)_{T\in \mathcal{T}_h}$ and
    $(\sigma(Z_T^{1}))_{T \in \mathcal{T}_h}$ as in \eqref{eq:aMC}; 
    \State Assemble the load vector $F_{MC}$ with entries
    $\big(F_{MC}(\varphi_{k})\big)_{k=1}^{N_h}$
    based on the values $(Z_T^2)_{T\in \mathcal{T}_h}$ and $(f(Z_T^{2}))_{T \in
    \mathcal{T}_h}$ as in \eqref{eq:FMC};
    \State Solve the linear equation $A_{MC} u_h^{MC}= F_{MC}$ to obtain
    $u_h^{MC}$; 
    \State \textbf{Output}: One realization of $u_h^{MC}$.  
  \end{algorithmic}
\end{algorithm}

Observe in Step~5 that one only has to sum over those triangles in
\eqref{eq:aMC}, which are contained in the joint support of the 
basis functions $\varphi_{k_1},\varphi_{k_2}$. Hence, the sum in \eqref{eq:aMC}
consists of at most two non-zero terms if $k_1\neq k_2$. 
In particular, the stiffness matrix $A_{MC}$ remains sparse and the complexity
of assembling $A_{MC}$ grows only linearly with $N_h$. In addition, the matrix
$A_h$ remains positive definite and allows the application of linear solvers
for large sparse systems as described in, e.g., \cite{hackbusch2016}.

\section{Numerical experiments}
\label{sec:numexp}

This section is devoted to some numerical experiments, which illustrate the
performance of the randomized quadrature formulas based on the 
MC estimator \eqref{eq:MC} and the IS estimator \eqref{eq:FIS}. 
To this end, we consider the Poisson equation \eqref{eq:Poisson} on
the domain $\mathcal{D}= (0, 1)^2 \subset \R^2$ with homogeneous Dirichlet
boundary conditions. In our experiments, we choose two different forcing terms:
The first is singular but still square-integrable. It is defined by 
\begin{align}
  \label{eqn:fterm}
  f_1(x,y):=|x-y|^{-q} + 10\sin(2^3 \pi x) \text{sgn}(2y-x), \quad \text{for }
  (x,y) \in \D,
\end{align}
with $q=0.49$ and $\mathrm{sgn} \colon \R \to \R$ given by
\begin{align*}
  \mbox{sgn}(x):=
  \begin{split}
    \begin{cases}
      -1,& \ \mbox{if\ }x<0,\\
       0,& \ \mbox{if\ }x=0,\\      
       1,& \ \mbox{if\ }x>0.
    \end{cases}
  \end{split}
\end{align*}
The second forcing term $f_2 \colon \D \to \R$ is taken more regular by
setting
\begin{align}
  \label{eqn:fterm2}
  f_2(x,y) := 8 x(1-x)y(1-y), \quad \text{for } (x,y) \in \D.
\end{align}
In fact, it can be easily verified that $f_2 \in H^1_0(\D) \cap H^2(\D)$.

For the finite element method we choose a family 
of structured uniform meshes. To be more precise, the domain $\D$ is first
subdivided into squares with uniform mesh size $h = 2^{-n}$, $n\in
\{2,\ldots,8\}$. Then, we obtain the triangulation $\mathcal{T}_h$ by bisecting
each square along the diagonal from the  upper left to the lower right vertex.
As in the previous sections, the shape functions are chosen 
to be piecewise linear. 
For each fixed triangulation $\mathcal{T}_h$ we then solve the discrete
problems \eqref{eq:discProb} and \eqref{eq:discProbis} as sketched 
in Algorithm~\ref{alg1}. As above, we denote the corresponding 
discrete solutions by $u^{MC}_h$ and $u^{IS}_h$, respectively.

To compare the performance of the two randomized quadrature formulas, 
we focus on the distances between the discrete solutions $u^{MC}_h$ and
$u^{IS}_h$ and the standard finite element solution $u_h = R_h u$, which
satisfies \eqref{eq:Galerkin} with $\sigma \equiv 1$ and the load vector $f_h$
defined in \eqref{eq:load}. 
This allows us to neglect the approximation error $u_h - u$ stemming from the
finite element method itself. More precisely, if the randomized
quadrature formulas are able to produce the exact values, e.g. $F_{IS}(\varphi_j) = [f_h]_j$ for every $j
\in \{1,\ldots,N_h\}$, then we immediately obtain $u_h^{IS} = u_h$. 

In the following, we therefore compute Monte Carlo approximations of the errors
$\|u^{MC}_h-u_h\|_{L^2(\Omega;H^1_0(\D))}$ and
$\|u^{IS}_h-u_h\|_{L^2(\Omega;H^1_0(\D))}$. This is achieved by
generating $M = 10^4$ independent realizations of the random variables
$u^{MC}_h$ and $u^{IS}_h$ and taking suitable averages. 
More precisely, we first take note of the fact that
$\E[u_h^{MC}] = \E[u_h^{IS}] = u_h$.
In fact, since $\sigma \equiv 1$ we have that 
$a_{MC} = a$ in \eqref{eq:discProb}. Hence, after 
taking expectation in \eqref{eq:discProb} and since $Q_{MC}$ is unbiased
we obtain that
\begin{align*}
  a\big( \E[ u_h^{MC} ], v_h \big)
  = \E\big[ a_{MC}(u_h^{MC}, v_h) \big] 
  = \E \big[ F_{MC}(v_h) \big]
  = F(v_h)
\end{align*}
for every $v_h \in S_h$. Therefore, the function $\E[u_h^{MC}] \in S_h$ is a
solution to \eqref{eq:Galerkin}, i.e. $\E[u_h^{MC}] = u_h$ for every $h \in
(0,1]$. The same arguments apply to $\E[u_h^{IS}]$. 

This motivates to replace $u_h$ in the error computation
by the Monte Carlo means 
\begin{align*}
  u_h \approx \frac{1}{M} \sum_{i = 1}^M u_{h,i}^{MC}, \quad \text{ and }
  u_h \approx \frac{1}{M} \sum_{i = 1}^M u_{h,i}^{IS},
\end{align*}
where $(u_{h,i}^{MC})_{i = 1}^M$ and $(u_{h,i}^{IS})_{i = 1}^M$ denote
families of independent and identically distributed copies of $u_h^{MC}$ and
$u_h^{IS}$, respectively.

The error based on the MC estimator with respect to the
$L^2(\Omega;H^1_0(\D))$-norm is then approximated by
\begin{align*}
  \|u^{MC}_h-u_h\|_{L^2(\Omega;H^1_0(\D))}^2
  &= \E \big[ | u_h^{MC} - \E[u_h^{MC}] |^2_{H^1(\D)} \big]\\
  &\approx \frac{1}{M-1} \sum_{i = 1}^{M} 
  \Big| u_{h,i}^{MC} - \frac{1}{M} \sum_{j = 1}^M
  u_{h,j}^{MC} \Big|^2_{H^1(\D)}.
\end{align*}
Observe that the estimator on the right-hand side coincides with
the empirical variance of the $S_h$-valued random samples $(u_{h,i}^{MC})_{i =
1}^M$. 

\begin{figure}[t]
  \begin{center}
    \subfigure[a][Errors of MC estimator with $f_1$.]{
      \includegraphics[width=0.47\textwidth]{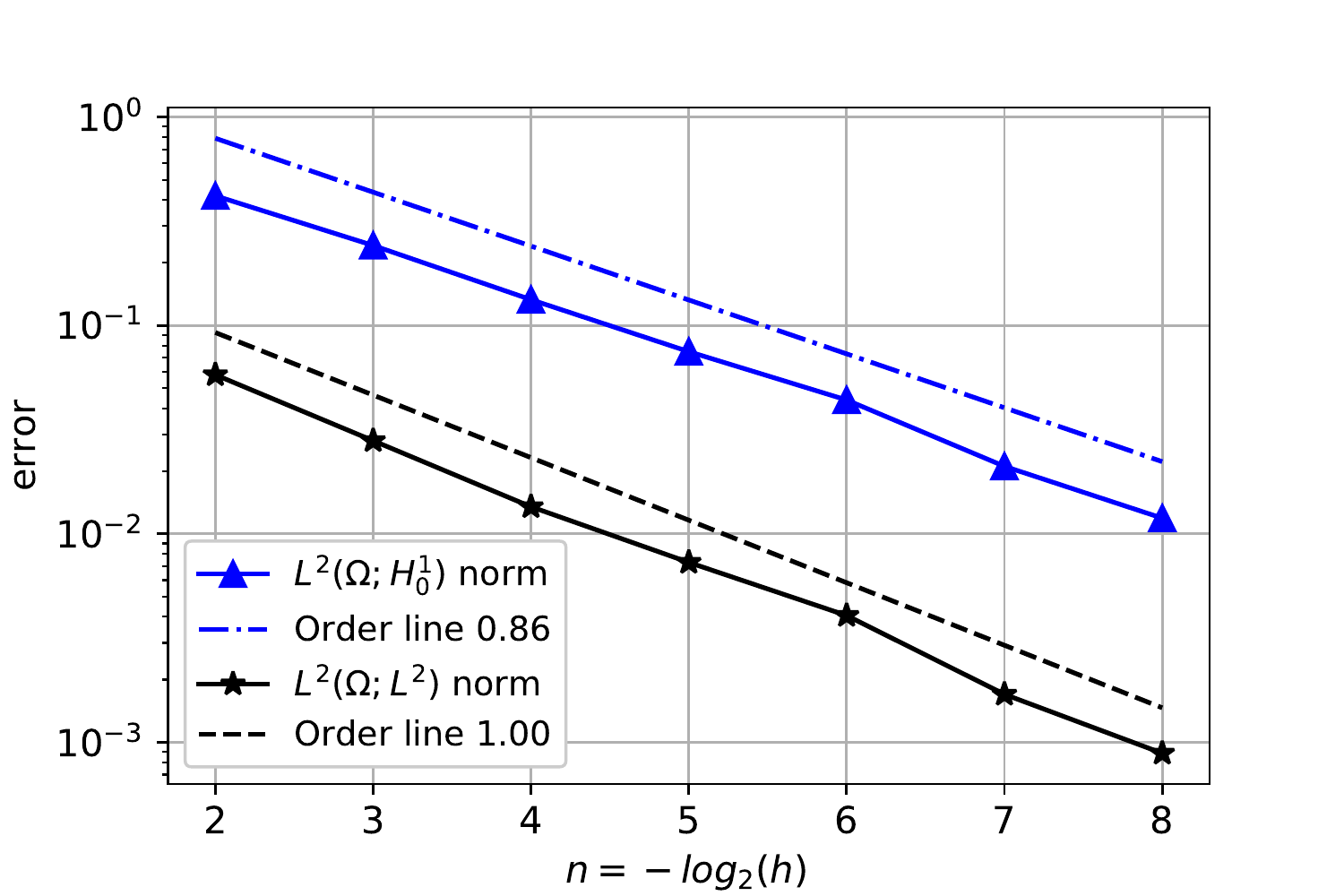}
    }
    \subfigure[b][Errors of MC estimator with $f_2$.]{
      \includegraphics[width=0.47\textwidth]{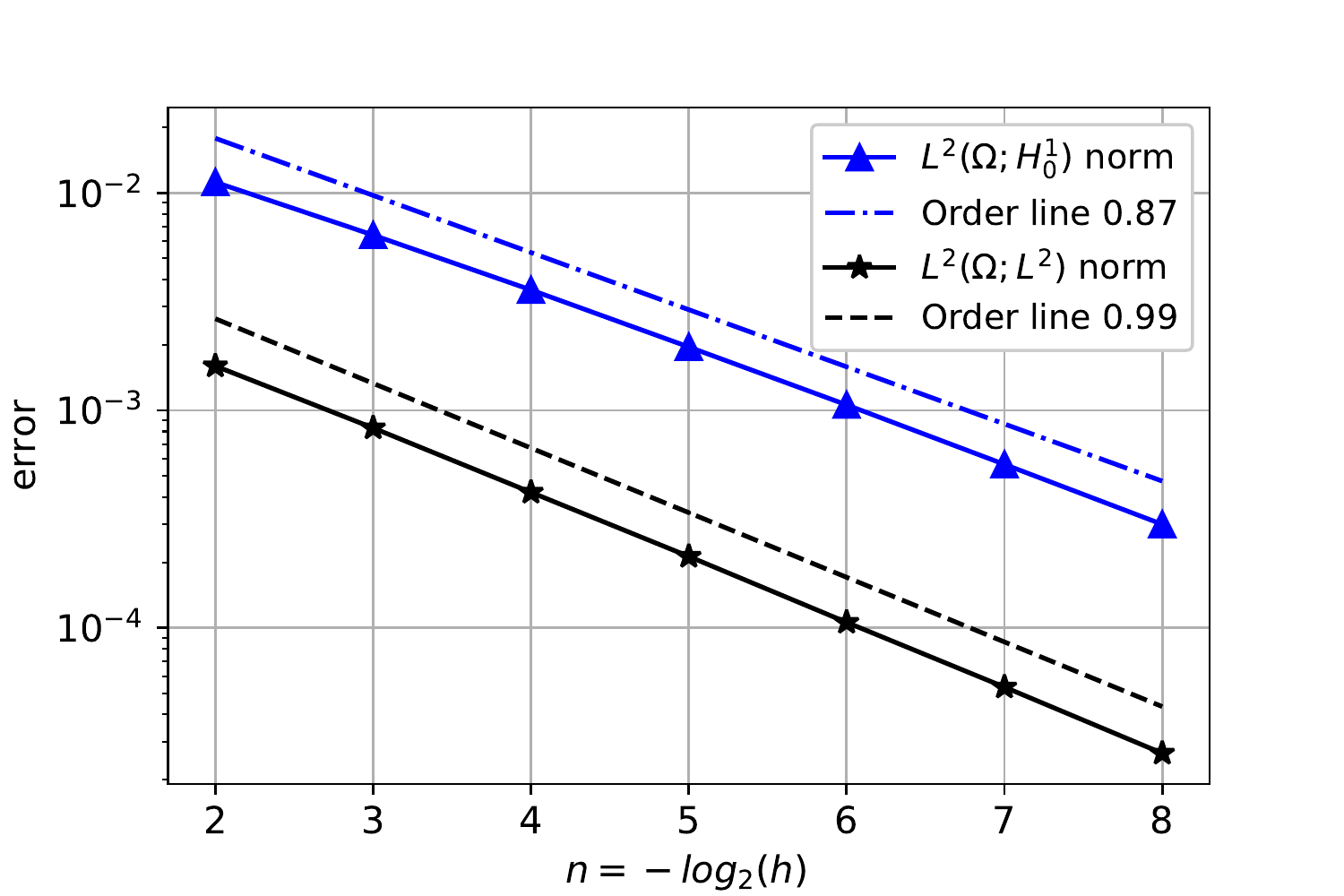}
    }
    \subfigure[c][Errors of IS estimator with $f_1$.]{
      \includegraphics[width=0.47\textwidth]{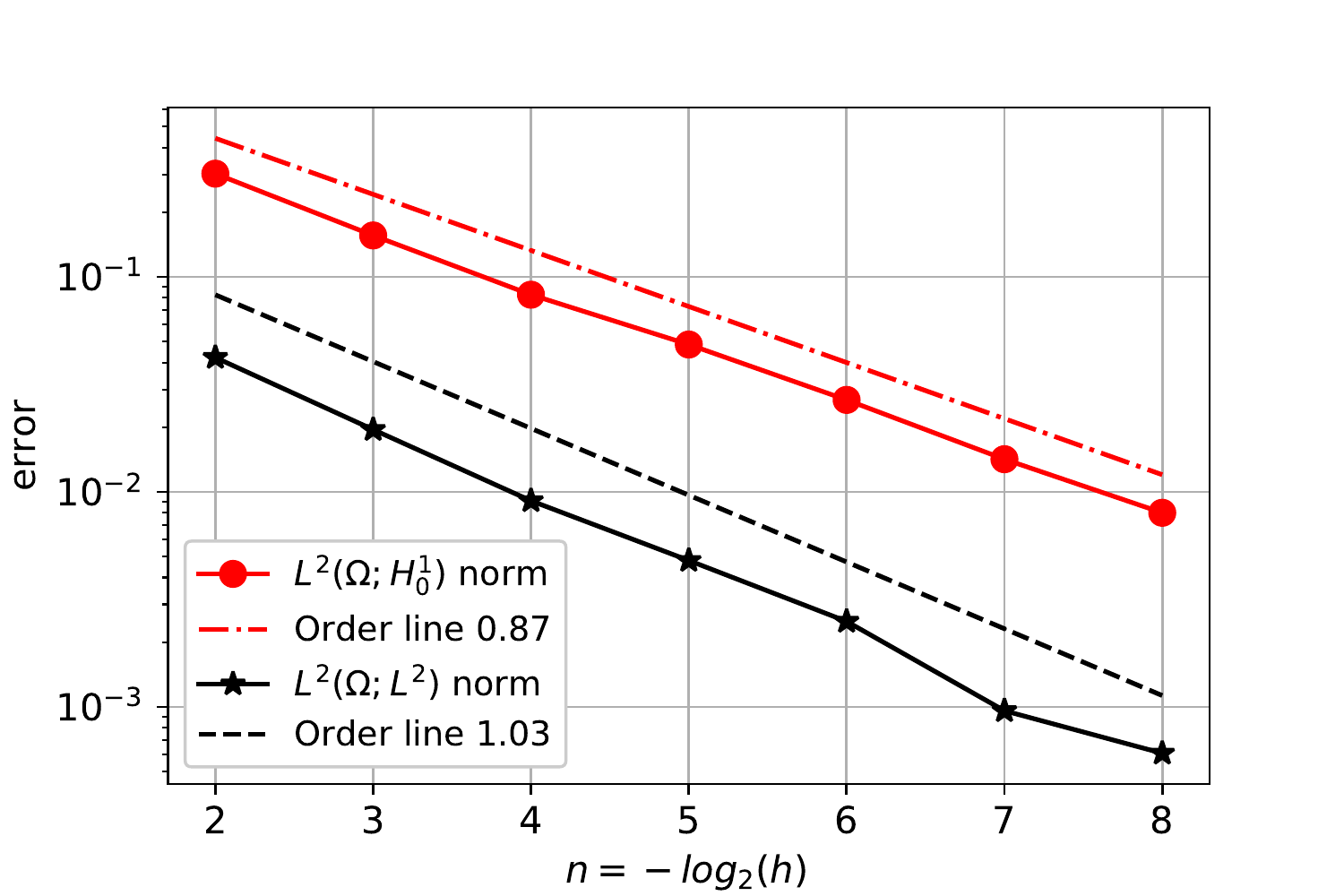}
    }
    \subfigure[d][Errors of IS estimator with $f_2$.]{
      \includegraphics[width=0.47\textwidth]{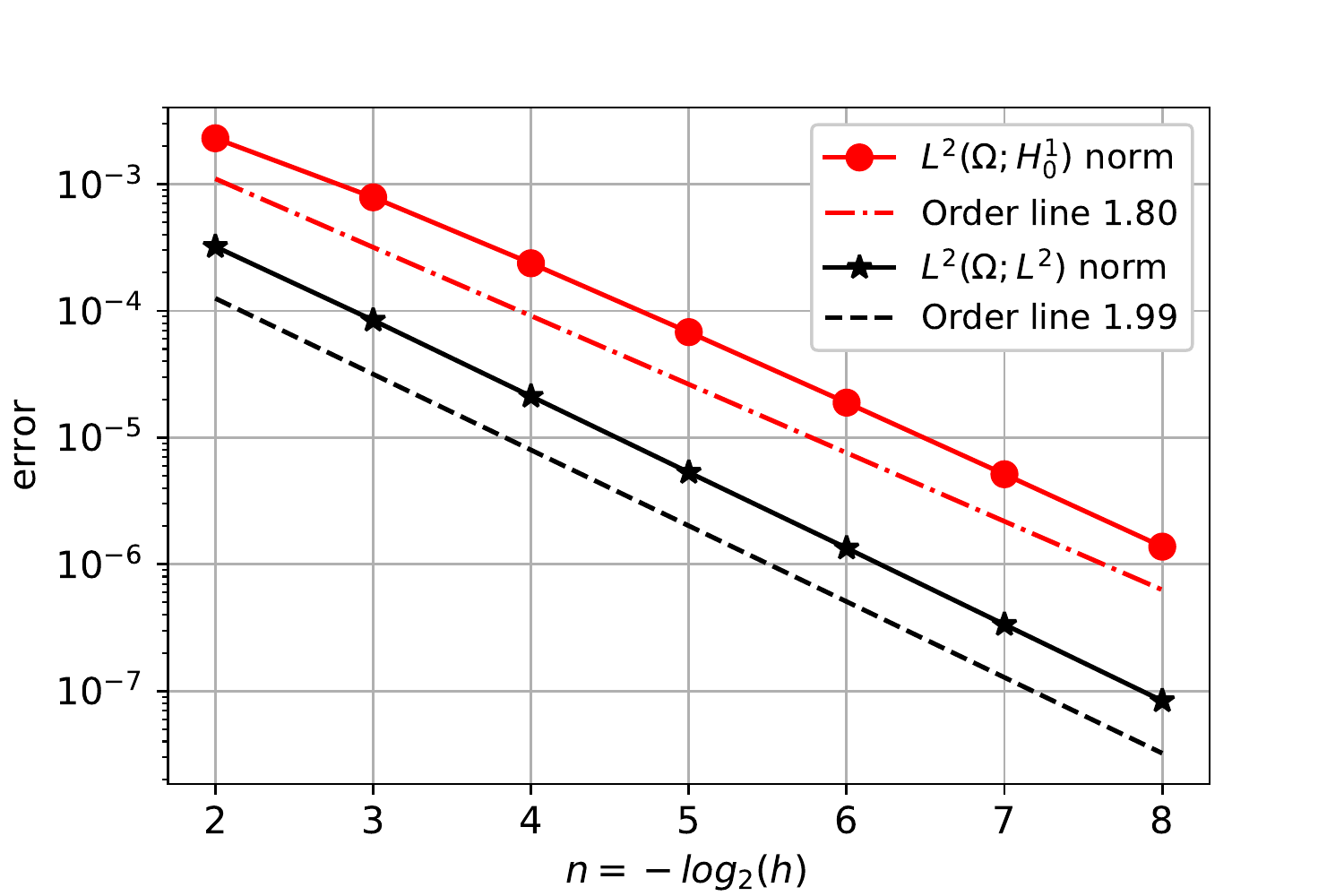}
    }
  \end{center}
  \caption{Error plots of the MC estimator \eqref{eq:MC} and IS estimator
  \eqref{eq:FIS} for the Poisson equation \eqref{eq:Poisson} with singular
  forcing term $f_1$ and smooth forcing term $f_2$.}
  \label{fig:MCvsIS}
\end{figure}

The approximation of $\|u^{IS}_h-u_h\|_{L^2(\Omega;H^1_0(\D))}$
is done in the same way. Further, we recall that the computation of the
$H^1(\D)$-semi-norm is easily accomplished in practice 
by making use of the relationship 
\begin{align}
  \label{eqn:approxHnorm1}
  |v_h|^2_{H^1(\D)} 
  = a(v_h,v_h) 
  = \sum_{i,j = 1}^{N_h} v_i v_j a( \varphi_i, \varphi_j)
  = \mathbf{v}^\top A_h \mathbf{v},
\end{align}
for every $v_h = \sum_{j = 1}^{N_h} v_j \varphi_j \in S_h$ with 
$\mathbf{v} = [v_1,\ldots,v_{N_h}]^{\top} \in \R^{N_h}$.
If the stiffness matrix $A_h$ is replaced by the mass matrix 
$M_h = [(\varphi_i,\varphi_j)_{L^2(\D)}]_{i,j=1}^{N_h}$
in \eqref{eqn:approxHnorm1}, then we also obtain
an approximation of the $L^2(\Omega;L^2(\D))$-norm.

Figure~\ref{fig:MCvsIS} shows the results of these experiments. 
In each of the four subfigures the Monte Carlo approximations of
the $L^2(\Omega;H^1_0(\D))$-norm and the $L^2(\Omega;L^2(\D))$-norm
of the errors $u_h^{MC} - u_h$ and $u_h^{IS} - u_h$ are plotted versus the mesh
size $h = 2^{-n}$, $n \in \{2,\ldots,8\}$. 
Hereby, the first two subfigures show the corresponding errors for the 
MC estimator \eqref{eq:MC} applied to the Poisson equation with the forcing
terms $f_1$ and $f_2$ defined in \eqref{eqn:fterm} and \eqref{eqn:fterm2},
respectively. As it can be seen from the order lines, the errors decay
approximately with orders roughly $0.86$ and $1$.
Given that $f_1$ is singular and only square-integrable, the experimental order
of convergence is therefore larger than it is predicted by
Theorem~\ref{thm:errorH1}.  

In Figures~\ref{fig:MCvsIS}~{\bf(c)} and {\bf(d)} we see the corresponding
results for the IS estimator \eqref{eq:FIS}. While the
values in Figure~\ref{fig:MCvsIS}~{\bf(c)} are comparable to those in 
Figure~\ref{fig:MCvsIS}~{\bf(a)}, it can be seen from
Figure~\ref{fig:MCvsIS}~{\bf(d)} that the IS estimator benefits 
considerably from the additional smoothness of $f_2$. In fact, the experimental
order of convergence is close to $2$ in Figure~\ref{fig:MCvsIS}~{\bf(d)}, which
is in line with the results in Theorem~\ref{thm:errorISL2}. 

\begin{figure}[t]
  \begin{center}
    \subfigure[a][Comparison for $f_1$.]{
      \includegraphics[width=0.47\textwidth]{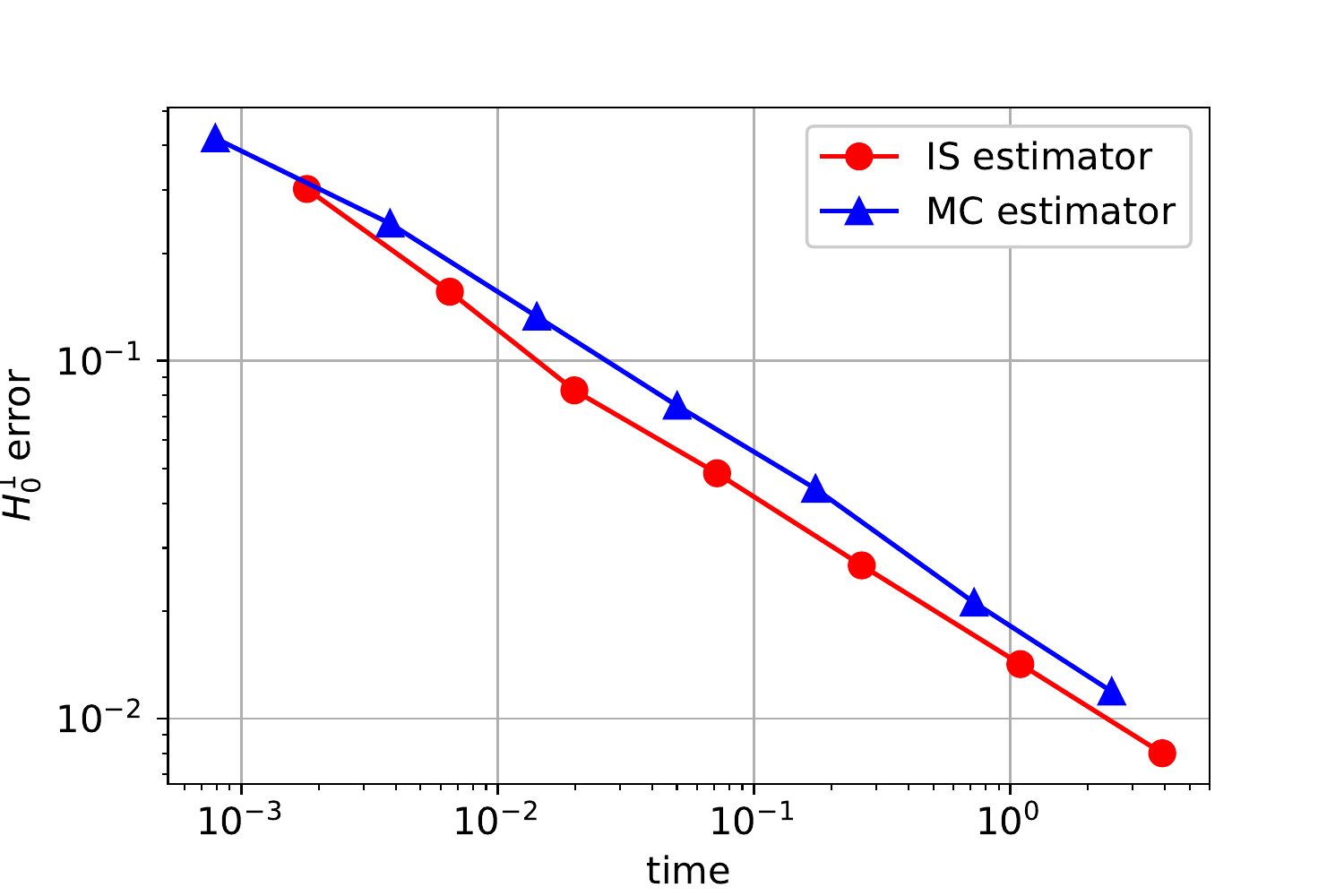}
    }
    \subfigure[b][Comparison for $f_2$.]{
      \includegraphics[width=0.47\textwidth]{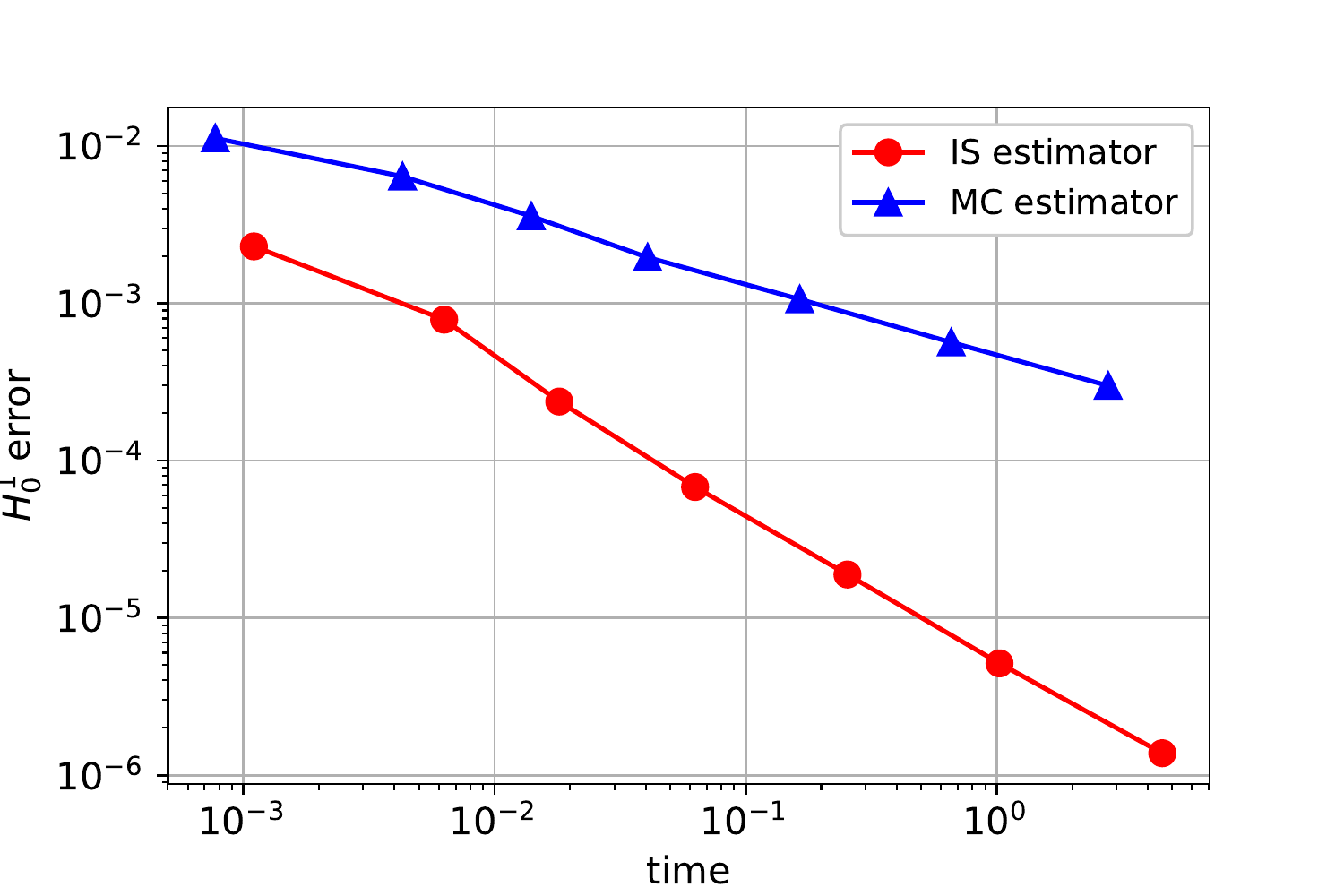}
    }
  \end{center}
  \caption{Computational time versus errors in $L^2(\Omega;H^1_0(\D))$-norm 
      of the MC estimator \eqref{eq:MC} and IS estimator 
      \eqref{eq:FIS} with singular forcing term $f_1$ and smooth forcing term
      $f_2$.} 
  \label{fig:time}
\end{figure}

In Figure~\ref{fig:time}, we plot the estimated values of the errors in
the $L^2(\Omega;H^1_0(\D))$-norm versus the computational time. 
This allows a better comparison of the performance of the two randomized
quadrature rules since the IS estimator is computational more expensive due to
the application of the general rejection method. 
Hereby, the computational time is taken as the average time needed to assemble
the load vector $f_h \in \R^{N_h}$ for $f_1$ or $f_2$ with either \eqref{eq:MC}
or \eqref{eq:FIS}. 
More precisely, we only measured the time of Step~6 in Algorithm~\ref{alg1}. 
The other steps are neglected, since they are essentially independent of the
choice of the randomized quadrature formula. 

As it can be seen in both subfigures, the importance sampling estimator
\eqref{eq:FIS} is superior to the MC estimator. For both forcing
terms the higher computational cost is offset by the better accuracy of the IS
estimator \eqref{eq:FIS}. In particular, this is true for the smooth forcing
term $f_2$ due to the better experimental order of convergence of
\eqref{eq:FIS}. On the other hand, it is not very pronounced for 
the singular forcing term $f_1$ as can be seen in
Figure~\ref{fig:time}~\textbf{(a)}. 

Finally, let us also briefly compare the performance of the randomized
quadrature formula with the deterministic \emph{barycentric quadrature rule},
which is also known as a one-point Gaussian quadrature formula. 
We refer to \cite{jin2015} and \cite[Section~5.6]{larsson2009}.
Table~\ref{tab1} lists the corresponding estimates of the errors 
stemming from the application of the deterministic quadrature rule.
Hereby, the errors are measured with respect to the semi-norm in
$H^1(\D)$. Apparently, the \emph{barycentric quadrature rule}
is not useful for approximating the load vector involving the
singular forcing term $f_1$. 

This is easily explained by the geometry of the triangulation $\mathcal{T}_h$.
For every mesh size $h = 2^{-n}$ there always exist triangles in 
$\mathcal{T}_h$ whose barycenters lie on the diagonal in $\D$, where $f_1$ is
singular. To avoid NaN entries in the load vector we replaced $f_1$ by the
modification 
\begin{align*}
  \tilde{f}_1(x,y) := (\mathrm{eps} + |x-y|)^{-q} + 10\sin(2^3 \pi x)
  \text{sgn}(2y-x), \quad \text{for } 
  (x,y) \in \D,  
\end{align*}
where $\mathrm{eps}$ is equal to the machine precision (in
\textsc{Matlab}\textsuperscript{\textcopyright} 
$\mathrm{eps} \approx 2.2204 \times 10^{-16}$). Nevertheless, the
discretization errors indicate that the \emph{barycentric quadrature rule}
is not reliable for applications with singular forcing terms.
This can only be circumvented by adapting the mesh to
avoid point evaluations close to singularities of the given forcing term. 
However, this requires a priori knowledge of the position of the singularities
or adaptive methods for their automatic detection when generating the mesh. The
randomized quadrature formulas, on the other hand, lead to a robustification of
the finite element method based on rudimentary uniform meshes without using any
preknowledge of the forcing term.

\begin{table}[t]
  \begin{center}
    \caption{Discretization errors of the (deterministic) barycentric
    quadrature rule applied to \eqref{eq:Poisson} with $f_1$.} 
    \label{tab1}
    \begin{tabular}{c|cccccc}
      \hline  
      \hline
      mesh size $h^{-n}$ & $n=3$ & $n=4$ & $n=5$ & $n=6$ & $n=7$ & $n=8$\\
      \hline
      error in $H^1_0(\D)$-norm
      & 1.4e+6 &7.7e+5 & 4.0e+5 &  2.1e+5 &1.0e+5 & 5.2e+4\\
      \hline
      \hline
    \end{tabular}
  \end{center}
\end{table}

\section*{Acknowledgement}

The authors like to thank Monika Eisenmann for helpful comments.
RK also gratefully acknowledges financial support by the German Research
Foundation through the research unit FOR 2402 -- Rough paths, stochastic
partial differential equations and related topics -- at TU Berlin. NP and YW
are grateful to EPSRC for funding this work through the project EP/R041431/1,
titled `Randomness: a resource for real-time analytics'. NP acknowledges
additional support from the Alan Turing Institute. YW would acknowledge EPSRC
project EP/S026347/1, titled 'Unparameterised multi-modal data, high order
signatures, and the mathematics of data science', as well as Alan Turing
Institute, for travel support.


\end{document}